\documentclass[11pt]{article}
\usepackage{amsfonts}
\usepackage{amsmath,color}
\usepackage[makeroom]{cancel}
\usepackage{mathtools}
\usepackage{bbm}
\usepackage[affil-it]{authblk}

\newcommand{\beq}{\begin{equation}}
\newcommand{\eeq}{\end{equation}}
\newcommand{\bea}{\begin{eqnarray}}
\newcommand{\eea}{\end{eqnarray}}
\newcommand{\beas}{\begin{eqnarray*}}
\newcommand{\eeas}{\end{eqnarray*}}

\newtheorem{theorem}{Theorem}[section]

\newtheorem{definition}[theorem]{Definition}
\newtheorem{proposition}[theorem]{Proposition}

\newtheorem{corollary}[theorem]{Corollary}
\newtheorem{lemma}[theorem]{Lemma}
\newtheorem{remark}[theorem]{Remark}
\newtheorem{example}[theorem]{Example}
\newtheorem{examples}[theorem]{Examples}
\newtheorem{foo}[theorem]{Remarks}

\newenvironment{proof}{\addvspace{\medskipamount}\par\noindent{\it
Proof}.}
{\unskip\nobreak\hfill$\Box$\par\addvspace{\medskipamount}}

\DeclareMathOperator\arctanh{arctanh}

\usepackage{hyperref}
\hypersetup{
    colorlinks=true,
    linkcolor=black,
    filecolor=magenta,      
    urlcolor=cyan,
    citecolor=blue
}

\newcommand{\bH}{\mathbb H}
\newcommand{\bS}{\mathbb S}

\newcommand{\R}{\mathbb R}

\newcommand{\A}{\mathfrak a}

\mathtoolsset{showonlyrefs=true}

\title{Quaternionic stochastic areas}
\author{Fabrice Baudoin\footnote{Author supported in part by the Simons Foundation}, Nizar Demni, Jing Wang }

\date{}

\begin{document}

\maketitle

\begin{abstract}
We  study quaternionic stochastic areas processes associated with  Brownian motions on the quaternionic rank-one symmetric spaces $\mathbb{H}H^n$ and $\mathbb{H}P^n$. The characteristic functions of fixed-time marginals of these processes are computed  and allows for the explicit description of their corresponding large-time limits. We also obtain exact formulas for the semigroup densities of the stochastic area processes using a Doob transform in the former case and the semigroup density of the circular Jacobi process in the latter. For $\mathbb{H}H^n$, the geometry of the quaternionic anti-de Sitter fibration plays a central role, whereas for $\mathbb{H}P^n$, this role is played by the quaternionic Hopf fibration.
\end{abstract}

\tableofcontents

\section{Introduction}

The goal of the paper is a thorough study of some functionals of the Brownian motion $(w(t))_{t \ge 0}$ on the  quaternionic spaces $\mathbb{H}H^n$ and $\mathbb{H}P^n$. Those functionals write as a  stochastic line integral
\[
\A (t)=\int_{w[0,t]} \zeta
\]
where $\zeta$ is a $\mathfrak{su}(2)$-valued one-form whose exterior derivative yields a.e. the quaternionic K\"ahler form of the underlying space. By analogy with our previous work \cite{SAWH}, we call those functionals quaternionic stochastic areas.

\

\textbf{Quaternionic stochastic area on  $\mathbb{H}^n$}

\

To motivate our study and present our approach in a simple situation, we first briefly comment on the case of the quaternionic flat space $\mathbb{H}^n$. More details about this case are worked out in Section 2. Let $\mathbb H$ be the non commutative field of quaternions and let $(w(t))_{t \ge 0}$ be a Brownian motion on $\mathbb H^n$, i.e. $(w(t))_{t \ge 0}$ is simply a $4n$-dimensional Euclidean Brownian motion. Consider the quaternionic stochastic area process defined by
\[
\A (t)=\int_{w[0,t]} \zeta=\frac{1}{2}\sum_{j=1}^n \int_0^t dw_j(s)\overline{w_j}(s)-  w_j(s)d\overline{w_j}(s) 
\]
where $\zeta=\frac{1}{2} \mathrm{Im} \langle q , dq \rangle:=\mathrm{Im}\sum_{i=1}^n q_i\overline{q'_i}$ is a $\mathfrak{su}(2) \simeq \mathbb{R}^3$ valued one-form. Following \cite{SAWH}, one can study the 3 dimensional process $\A$ by embedding it into a higher dimensional Markov process. More precisely, the $4n+3$ dimensional process
\[
(X_t)_{t \geq 0} = (w(t), \A (t))_{t \geq 0},
\]
is a Markov process and its generator is the sub-Laplacian on the quaternionic Heisenberg group. Accordingly, $\A$ can be interpreted as the fiber motion of the horizontal Brownian on the quaternionic Heisenberg group. This interpretation, together with a skew-product decomposition of this horizontal Brownian motion, readily yields the identity in distribution 
 \[
\left(\A (t) \right)_{t \ge 0} \overset{d}{=} \left(\beta_{\frac14\int_0^t  r^2(s)ds}\right)_{t \ge 0},
\]
where $(\beta_t)_{t \ge 0}$ is a standard $3$-dimensional Brownian motion independent from the $4n$-dimensional Bessel process $r(t)=| w (t)|, t \geq 0$. One then deduces from \cite{Yor} an exact formula for the characteristic function of $\A (t)$ and deduce then, by Fourier inversion, an integral formula for the density.

\newpage

\textbf{Quaternionic stochastic area on  $\mathbb{H}H^n$}

\

The method described for the quaternionic flat space $\mathbb{H}^n$  extends to the case of the quaternionic hyperbolic space $\mathbb{H}H^n$. If $(w(t))_{t \ge 0}$ is now the Riemannian Brownian motion on $\mathbb H H^n$, then the functional of interest writes
\[
\A(t)=\int_{w[0,t]} \zeta=\frac{1}{2}\sum_{j=1}^n \int_0^t \frac{dw_j(s)\overline{w_j}(s)-  w_j(s)d\overline{w_j}(s)}{1-|w(s)|^2},
\]
where $\zeta$ is still a $\mathfrak{su}(2)$-valued one-form and $(w_1,\cdots,w_n)$ are now the inhomogeneous coordinates on $\mathbb{H}H^n$. Indeed, Theorem \ref{horizon} below describes the stochastic area  process $\A$ in terms of the fiber motion of the horizontal Brownian motion of the  quaternionic anti de-Sitter fibration
\[
\mathbf{SU}(2)\to \mathbf{AdS}^{4n+3}(\mathbb{H})\to \bH H^n.
\]
The geometry of this fibration therefore plays a prominent role in the study of $\A$, which has been studied along with its related heat kernels in the paper \cite{BDW}. In this framework, one also obtains the following identity in distribution 
 \[
\left(\A (t) \right)_{t \ge 0} \overset{d}{=} \left(\beta_{\int_0^t  \tanh r^2(s)ds}\right)_{t \ge 0},
\]
where $(\beta_t)_{t \ge 0}$ is a standard $3$-dimensional Brownian motion and is independent from the radial process $(r(t)=| w (t)|)_{t \geq 0}$. The latter one  is now a hyperbolic Jacobi process. Using the methods developed in \cite{SAWH}, one can then compute the characteristic function of $\A (t)$ and deduce that when $t \to +\infty$, the following convergence in distribution takes place
\[
\frac{\A(t)}{\sqrt{t}} \to \mathcal{N}(0,\mathrm{Id}_3)
\]
where $\mathcal{N}(0,\mathrm{Id}_3)$ is $3$-dimensional normal distribution with mean 0 and variance the identity matrix.

\

\textbf{Quaternionic stochastic area on  $\mathbb{H}P^n$}

\

Another geometry for which our previous reasoning also applies is that of the quaternionic projective space $\mathbb{H}P^n$, which is the positively-curved analog of $\mathbb{H}H^n$. Let $(w(t))_{t \ge 0}$ the Riemannian Brownian motion on $\mathbb H P^n$, then the corresponding generalized stochastic area process is defined by:
\[
\A(t) := \int_{w[0,t]} \zeta=\frac{1}{2}\sum_{j=1}^n \int_0^t \frac{dw_j(s)\overline{w_j}(s)-  w_j(s)d\overline{w_j}(s)}{1+|w(s)|^2}
\]
where we still denote (and hope there is no confusion) $(w_1,\cdots,w_n)$ the inhomogeneous coordinates on $\mathbb{H}P^n$. This time, Theorem \ref{horizon-S} describes $\A$ by means of the fiber motion of the horizontal Brownian motion of the  quaternionic Hopf fibration
\[
\mathbf{SU}(2)\to \mathbb{S}^{4n+3}(\mathbb{H})\to \bH P^n.
\]
Similarly, we shall prove the identity in distribution 
 \[
\left(\A (t) \right)_{t \ge 0} \overset{d}{=} \left(\beta_{\int_0^t  \tan r^2(s)ds}\right)_{t \ge 0},
\]
where $(\beta_t)_{t \ge 0}$ is again a standard $3$-dimensional Brownian motion independent from the radial process $(r(t) := |w (t)|)_{t \geq 0}$ which is a circular Jacobi process. As before, we are able to compute the characteristic function of $\A (t)$ and describe its large-time limit. As a consequence, we prove that the following convergence in distribution takes place
\[
\frac{\A(t)}{\sqrt t} \to \mathcal{N}(0, 2n \mathrm{Id}_3).
\]

\section{Preliminary: Stochastic area process on the quaternionic space}

In this preliminary section, we recall some results about the 3-dimensional stochastic area process associated with a $4n$ dimensional Euclidean Brownian motion. Stochastic area processes associated to Euclidean Brownian motions and their related distributions are well understood and have been extensively studied in the literature, see for instance \cite{Gaveau, Levy, Yor}. However, our goal here is to highlight in this simple situation the role of quaternionic geometry and to present the structural ideas that will  be used in later sections.

Let $\mathbb{H}$ be the non-commutative field of quaternions
\[
\mathbb{H}=\{q=t+xI+yJ+zK, (t,x,y,z)\in\R^4\},
\]
where  $I,J,K$ are the Pauli matrices given by
\[
I=\left(
\begin{array}{ll}
i & 0 \\
0&-i 
\end{array}
\right), \quad 
J= \left(
\begin{array}{ll}
0 & 1 \\
-1 &0 
\end{array}
\right), \quad 
K= \left(
\begin{array}{ll}
0 & i \\
i &0 
\end{array}
\right).
\]
For $q=t+xI+yJ+zK \in \mathbb{H}$, we denote by $\overline q= t -xI-yJ-zK$ its conjugate, $|q|^2=t^2+x^2+y^2+z^2$ its squared norm and $\mathrm{Im}(q)=(x,y,z) \in \mathbb{R}^3$ its imaginary part.
The quaternionic Heisenberg group is then defined as the product space $\mathbf{H}^{4n+3}(\bH)=\bH^{n}\times\mathrm{Im}(\bH)$ with the group law
\[
(q,\phi )*(q', \phi')=\left(q+q',\phi+\phi'+\frac12\mathrm{Im}\langle q,q'\rangle \right)
\]
 where for $q = (q_1, \dots, q_n), q' = (q'_1, \dots, q'_n) \in \bH^{n}$, we have set $\mathrm{Im}\langle q,q'\rangle=\mathrm{Im}\sum_{i=1}^n q_i\overline{q'_i}$\footnote{Note that the different convention that $\mathrm{Im}\langle q,q'\rangle=\mathrm{Im}\sum_{i=1}^n \overline{q_i} q'_i$ is also sometimes used in the literature, for instance in \cite{Cow}.}. This is an example of H-type groups (see \cite{Cow}) which play an important role in sub-Riemannian geometry (see 	\cite{BGMR}).
 
 If $\phi = \phi_I I + \phi_J J + \phi_K K$, then the right invariant vector fields 
 \begin{align*}
 &V^1_i=\frac{\partial}{\partial t_i}-\frac12\left(x_i\frac{\partial}{\partial \phi_I}+y_i\frac{\partial}{\partial \phi_J}+z_i\frac{\partial}{\partial \phi_K} \right)\\
 &V^I_i=\frac{\partial}{\partial x_i}+\frac12\left(t_i\frac{\partial}{\partial \phi_I}+z_i\frac{\partial}{\partial \phi_J}-y_i\frac{\partial}{\partial \phi_K} \right)\\
 &V^J_i=\frac{\partial}{\partial y_i}+\frac12\left(-z_i\frac{\partial}{\partial \phi_I}+t_i\frac{\partial}{\partial \phi_J}+x_i\frac{\partial}{\partial \phi_K} \right)\\
  &V^K_i=\frac{\partial}{\partial z_i}+\frac12\left(y_i\frac{\partial}{\partial \phi_I}-x_i\frac{\partial}{\partial \phi_J}+t_i\frac{\partial}{\partial \phi_K} \right)
 \end{align*}
 together with the fiber vector fields
 \[
 T_I=\frac{\partial}{\partial \phi_I},\quad  T_J=\frac{\partial}{\partial \phi_J}, \quad  T_K=\frac{\partial}{\partial \phi_K},
 \] 
 generate the Lie algebra of $\mathbf{H}^{4n+3}(\bH)$. The sub-Laplacian on $\mathbf{H}^{4n+3}(\bH)$ is then given by
  \begin{align*}
 \Delta_{\mathbf{H}^{4n+3}(\bH)}=&\sum_{i=1}^n (V^1_i)^2+(V^I_i)^2+(V^J_i)^2+(V^K_i)^2\\
&= \Delta_{\R^{4n}}+\frac{|q|^2}{4}\left(\frac{\partial^2}{\partial \phi_I^2}+\frac{\partial^2}{\partial \phi_J^2}+\frac{\partial^2}{\partial \phi_K^2}\right)+\sum_{i=1}^n\sum_{S=I, J,K}\left(q_i\frac{\partial}{\partial q_i}S-S\frac{\partial}{\partial \overline{q_i}}\overline{q_i} \right)\frac{\partial}{\partial\phi_S},\\
&=\Delta_{\R^{4n}}+\frac{|q|^2}{4}\Delta_{\R^3}+\sum_{i=1}^n\sum_{S=I, J,K}\left(q_i\frac{\partial}{\partial q_i}S-S\frac{\partial}{\partial \overline{q_i}}\overline{q_i} \right)\frac{\partial}{\partial\phi_S}
 \end{align*}
 where for each $q_i=t_i+x_iI+y_iJ+z_iK$, we set 
 \begin{equation*}
 \frac{\partial}{\partial q_i}:=\frac12 \left(\frac{\partial}{\partial t_i}- \frac{\partial}{\partial x_i}I- \frac{\partial}{\partial y_i}J- \frac{\partial}{\partial z_i}K\right).
 \end{equation*}

What makes the connection between stochastic areas and the quaternionic Heisenberg group is that the operator $\frac{1}{2} \Delta_{\mathbf{H}^{4n+3}(\bH)}$ is the generator of the Markov process:
\begin{equation}\label{eq-SA-He}
\left( X_t\right)_{t \geq 0} = \left(B_1(t), \dots, B_n(t), \frac12\sum_{i=1}^n\int_0^t\mathrm{Im}\langle dB_i(s), B_i(s) \rangle\right)_{t \geq 0},
\end{equation}
where $(B_i(t))_{t \geq 0}, 1 \leq i \leq n$, are independent $\bH$-valued Brownian motions. Equivalently, $(X_t)_{t \ge 0}$ is the horizontal Brownian motion of the canonical sub-Riemannian structure on $\mathbf{H}^{4n+3}(\bH)$.
 
Then, the Euclidean norm $r:=|B|$ is the radial process of $B = (B_1, \dots, B_n)$ while the $\R^3$-valued fiber motion is given by the stochastic area process: 
\begin{equation*}
\phi(t):=\frac12\sum_{i=1}^n\int_0^t\mathrm{Im}\langle dB_i(s), B_i(s)\rangle, \quad t \geq 0.
\end{equation*} 
 
The process $\left( r(t), \phi(t)\right)_{t \ge 0}$ is a diffusion with generator
 \[
L=\frac{1}{2} \left(\frac{\partial^2}{\partial r^2}+\frac{4n-1}{r}\frac{\partial}{\partial r}+\frac{r^2}{4}\Delta_{\R^3}\right).
 \]
 Consequently, the following equality in distribution holds:
 \[
\left( r(t) ,\phi(t) \right)_{t \ge 0} \overset{d}{=}  \left( r(t),\beta_{\frac14\int_0^t  r^2(s)ds}\right)_{t \ge 0},
\]
where $(\beta_t)_{t \ge 0}$ is a standard $3$-dimensional Brownian motion independent from $r$.

The characteristic function of $\phi(t)$ may be derived from the computations done by M. Yor in \cite{Yor}. More precisely, let $\lambda = (\lambda_I, \lambda_J, \lambda_K) \in [0,\infty)^3$, $r \in [0,\infty),$ and consider
\[
I(\lambda,r) := \mathbb{E}\left(e^{i \lambda\cdot \phi(t)}\mid r(t)=r\right)  =\mathbb{E}\left(e^{- \frac{|\lambda|^2}{8} \int_0^t r^2(s)ds}\mid r(t)=r\right) 
\]
where $|\lambda|^2=\lambda_I^2+\lambda_J^2+\lambda_K^2$. From \cite{Yor}, it is known that:
\[
\mathbb{E}\left(e^{- \frac{|\lambda|^2}{8} \int_0^t r^2(s)ds}\mid r(t)=r\right) =\left(\frac{|\lambda|t}{2\sinh \left(\frac{|\lambda|t}2\right)} \right)^{2n}e^{-\frac{r^2}{2t}\left(\frac{|\lambda|t}2\coth\left(\frac{|\lambda|t}2\right)-1\right)},
\]
whence it follows that: 
\begin{equation}\label{eq-lim-He}
\mathbb{E}\left(e^{i \lambda\cdot \phi(t)}\right) 
=\mathbb{E}\left(e^{- \frac{|\lambda|^2}{8} \int_0^t r^2(s)ds}\right) =\left(\cosh \frac{|\lambda|t}2\right)^{-2n}.
\end{equation}

As a matter of fact, the distribution of $\phi(t)$ is a 3-dimensional analog of the Meixner distribution (\cite{Sch}) which, up to our best knowledge, has never appeared in literature. 

\begin{proposition}
The density of $\phi(t), t>0,$ with respect to the Lebesgue measure is given by
\begin{equation*}
h_t(\phi) = \frac{2^{2n-1}}{2\pi t^3} \int_{-1}^1 du \int_0^1 dv [v(1-v)]^{n-1} \left[\frac{v}{1-v}\right]^{iu|\phi|/t}\ln^2\left[\frac{v}{1-v}\right] .
\end{equation*}
\end{proposition}

\begin{proof}
Using  Fourier inversion formula, the density of $\phi(t)$ is given up to a normalizing constant by: 
\begin{equation*}
h_t(\phi) := \int_{\mathbb{R}^3}\left(\cosh \frac{|\lambda|t}2\right)^{-2n}e^{-i \lambda\cdot \phi} d\lambda,
\end{equation*}
which reads in polar coordinates: 
\begin{equation*}
h_t(\phi) = \int_{0}^{\infty}r^2\left(\cosh \frac{rt}2\right)^{-2n}\left\{\int_{S^2} e^{-i r \theta \cdot \phi} d\theta\right\} dr, \quad \phi \in \mathbb{R}^3.
\end{equation*}
By rotation invariance, the inner integral may be written as: 
\begin{equation*}
\int_{-1}^1e^{-i r u |\phi|} du = \int_{-1}^1\cos(r u |\phi|)du,
\end{equation*}
and as such, Fubini Theorem entails: 
\begin{equation*}
h_t(\phi) = \int_{-1}^1 \int_{0}^{\infty}r^2\cos(r u |\phi|) \left(\cosh \frac{rt}2\right)^{-2n} dr du.
\end{equation*}
But, we know from table 3.985 in \cite{Gra-Ryz} that for any $z \in \mathbb{R}$: 
\begin{align*}
\int_{0}^{\infty} \cos(r z) \left(\cosh \frac{rt}2\right)^{-2n} dr & = \frac{2^{2n-1}}{t(2n-1)!}\left|
\Gamma\left(n+i\frac{z}{t}\right)\right|^2 
\\& = \frac{2^{2n-1}}{t}\int_0^1 [v(1-v)]^{n-1} \left[\frac{v}{1-v}\right]^{iz/t}dv,
\end{align*}
whence
\begin{align*}
\int_{0}^{\infty}r^2\cos(rz) \left(\cosh rt/2\right)^{-2n} dr  & = -\frac{2^{2n-1}}{t(2n-1)!} \frac{d^2}{dz^2}\left|
\Gamma\left(n+i\frac{z}{t}\right)
\right|^2
\\& = \frac{2^{2n-1}}{t^3}\int_0^1 [v(1-v)]^{n-1} \left[\frac{v}{1-v}\right]^{iz/t}\ln^2\left[\frac{v}{1-v}\right] dv.
\end{align*}
Note that the last integral is absolutely convergent (uniformly in $z$) which may be easily seen after performing there the variable change $y = v/(1-v)$: 
\begin{equation*}
\int_0^1 [v(1-v)]^{n-1} \left[\frac{v}{1-v}\right]^{iz/t}\ln^2\left[\frac{v}{1-v}\right] dv = \int_0^{\infty} y^{iz/t} \frac{y^{n-1}}{(1+y)^{2n}} \ln^2(y) dy.
\end{equation*}
Substituting $z = u|\phi|$ and taking into account the factor $1/(2\pi)$ present in Fourier inversion formula, the density follows.  
\end{proof}

%
%
%

\section{Stochastic area process on the quaternionic hyperbolic space $\bH H^n$}\label{sec-geometry}

\subsection{Quaternionic anti-de Sitter fibration}

We first give a quick overview of the quaternionic anti-de Sitter fibration but refer to \cite{BDW} and the references therein for further details. Recall the quaternionic field is defined by
\[
\mathbb{H}=\{q=t+xI+yJ+zK, (t,x,y,z)\in\R^4\},
\]
where  $I,J,K \in \mathbf{SU}(2)$ are the Pauli matrices. Then, the quaternionic anti-de Sitter space $\mathbf{AdS}^{4n+3}(\mathbb{H})$ is defined as the quaternionic pseudo-hyperboloid:
\[
\mathbf{AdS}^{4n+3}(\mathbb{H})=\lbrace q=(q_1,\cdots,q_{n+1})\in \mathbb{H}^{n+1}, \| q \|^2_H =-1\rbrace,
\]
where 
\[
\|q\|_H^2 := \sum_{k=1}^{n}|q_k|^2-|q_{n+1}|^2.
\]

The group $\mathbf{SU}(2)$, viewed as the set of unit quaternions, acts isometrically on $\mathbf{AdS}^{4n+3}(\mathbb{H})$ and the quotient space 
\begin{equation*}
\mathbf{AdS}^{4n+3}(\mathbb{H})/ \mathbf{SU}(2)
\end{equation*}
can be identified with the quaternionic hyperbolic space $\bH H^n$ endowed with its canonical quaternionic K\"ahler metric. 
The projection map $\mathbf{AdS}^{4n+3}(\mathbb{H})\to \bH H^n$ is a pseudo-Riemannian submersion with totally geodesic fibers isometric to $\mathbf{SU}(2)$. The fibration
\[
\mathbf{SU}(2)\to \mathbf{AdS}^{4n+3}(\mathbb{H})\to \bH H^n
\]
is referred to as the quaternionic anti de-Sitter fibration. 

As in \cite{BDW}, we shall work with cylindrical coordinates on $\mathbf{AdS}^{4n+3}(\mathbb{H})$. Let $(w_1,\dots,w_n)$ denote a point on the base space $\bH H^n$, and $(\theta_I, \theta_J, \theta_K)$ be the coordinates in the Lie algebra $\mathfrak{su}(2)$ of traceless skew-Hermitian $2\times 2$ matrices. More precisely, we shall consider the map
\begin{eqnarray*}
\bH H^n \times \mathfrak{su}(2) & \rightarrow & \mathbf{AdS}^{4n+3}(\mathbb{H}) \\
(w_1,\dots, w_n, \theta_I, \theta_J, \theta_K) & \mapsto &\left(\frac{e^{ I\theta_I +J\theta_J +K\theta_K} w_1}{\sqrt{1-\rho^2}},\cdots,\frac{e^{ I\theta_I +J\theta_J +K\theta_K}w_n}{\sqrt{1-\rho^2}},\frac{e^{ I\theta_I +J\theta_J +K\theta_K}}{\sqrt{1-\rho^2}} \right)
\end{eqnarray*}
where $\rho = \sqrt{\sum_{j=1}^{n}|w_j|^2}$ and $w_i = q_{n+1}^{-1}q_i$, $i=1,\dots, n,$ are inhomogeneous coordinates in $\bH H^n$.

 \subsection{Quaternionic stochastic area process on $\mathbb{H}H^n$}

We define the quaternionic stochastic area process as follows.

\begin{definition}
Let $(w(t))_{t \ge 0}$ be a Brownian motion on $\mathbb{H}H^n$ started at $0$\footnote{We call $0$ the point with inhomogeneous coordinates $w_1=0,\cdots, w_{n}=0$}. The quaternionic stochastic area process of $(w(t))_{t \ge 0}$ is the process in $\mathfrak{su}(2) \simeq \R^3$ defined by
\[
\A(t):=\int_{w[0,t]} \zeta=\frac{1}{2}\sum_{j=1}^n \int_0^t \frac{dw_j(s)\overline{w_j}(s)-  w_j(s)d\overline{w_j}(s)}{1-|w(s)|^2},
\]
where the above stochastic integrals are understood in the Stratonovich, or equivalently in the It\^o sense.
\end{definition}

The following theorem shows that the quaternionic stochastic area process of the Brownian motion on $\mathbb{H}H^n$ can be interpreted as the fiber motion of the horizontal Brownian motion on $\mathbf{AdS}^{4n+3}(\mathbb{H})$.

 \begin{theorem}\label{horizon}
 Let $(w(t))_{t \ge 0}$ be a Brownian motion on $\mathbb{H}H^n$ started at 0, and $(\Theta(t))_{t\ge 0}$ be the $\mathbf{SU}(2)$-valued process solution of the Stratonovitch stochastic differential equation
 \begin{equation}\label{eq-Maurer-Cartan}
 d\Theta(t) =\Theta(t)  \circ d\A(t),
 \end{equation}
where we identify $\A$ as an element of the Lie algebra $\mathfrak{su}(2)$. Then, the $\mathbf{AdS}^{4n+3}(\mathbb{H})$-valued diffusion process
 \begin{equation}\label{eq-BM-H}
 X(t)=\frac{{\Theta(t)} }{\sqrt{1-|w(t)|^2}} \left( w(t),1\right), \quad t \ge 0
 \end{equation}
 is the horizontal stochastic lift at the north pole of $(w(t))_{t \ge 0}$ by the submersion $\mathbf{AdS}^{4n+3}(\mathbb{H}) \to \bH H^n$.
 \end{theorem}
\begin{remark}\label{remark-MC}
The SDE \eqref{eq-Maurer-Cartan} implies that the integral of the Maurer-Cartan form on $\mathbf{SU}(2)$ along $\Theta(t)$ is exactly given by the stochastic area process $(\A(t))_{t \geq 0}$.
\end{remark}
%

 \begin{proof}
We use the fact that the submersion $\pi$ is compatible with the quaternionic contact structure of $\mathbf{AdS}^{4n+3}(\mathbb{H})$ which is described in Section \ref{appendix 1}. Precisely, the horizontal distribution of this submersion is the kernel of the contact form $\Lambda$ given by \eqref{Lambda} and the fibers of the submersion are the orbits of the Reeb vector fields.  We claim that the horizontal lift to $\mathbf{AdS}^{4n+3}(\mathbb{H})$  of the vector field 
\begin{equation*}
\frac{\partial}{\partial w_i}:=\frac12\left(\frac{\partial}{\partial t_i}-\frac{\partial}{\partial x_i}I-\frac{\partial}{\partial y_i}J-\frac{\partial}{\partial z_i}K\right)
\end{equation*}
is given by: 
\begin{equation}\label{eq-V-H}
V_i= \frac{\partial}{\partial w_i} + \frac{\overline{w_i}}{2(1-\rho^2)\cos^2\eta}\frac{\partial}{\partial \phi},
\end{equation}
where we set
\begin{equation}\label{eq-phi}
\phi:=\frac{\tan\eta}{\eta}\mathfrak{q}=\phi_II+\phi_JJ+\phi_KK,\quad 
\frac{\partial}{\partial \phi} := \frac{\partial}{\partial \phi_I} I + \frac{\partial}{\partial \phi_J} J+ \frac{\partial}{\partial \phi_K}K.
\end{equation}

The derivation of $V_i$ is as follows. The contact form 
\[
\Lambda=\zeta-\cos^2\eta\,d\phi
\] is in fact a $3$-dimensional one-form, where  $\zeta$ is defined in \eqref{QKah} from the appendix and $d\phi=d\phi_II+d\phi_JJ+d\phi_KK$. Writing $\zeta=\zeta^II+\zeta^JJ+\zeta^KK$ with
\begin{align*}
&(1-\rho^2)\zeta^I=t_idx_i-x_idt_i+y_idz_i-z_idy_i\\
&(1-\rho^2)\zeta^J=t_idy_i-y_idt_i+z_idx_i-x_idz_i\\
&(1-\rho^2)\zeta^K=t_idz_i-z_idt_i+x_idy_i-y_idx_i
\end{align*}
then we get:
\[
2(1-\rho^2)\zeta^I( \frac{\partial}{\partial w_i})=-\overline{w}_iI, \quad 2(1-\rho^2)\zeta^J( \frac{\partial}{\partial w_i})=-\overline{w}_iJ,\quad 2(1-\rho^2)\zeta^K( \frac{\partial}{\partial w_i})=-\overline{w}_iK,
\]
\[
2(1-\rho^2)\zeta^I( \frac{\partial}{\partial \overline{w}_i})=Iw_i, \quad 2(1-\rho^2)\zeta^J( \frac{\partial}{\partial \overline{w}_i})=Jw_i,\quad 2(1-\rho^2)\zeta^K( \frac{\partial}{\partial \overline{w}_i})=Kw_i.
\]
Now, the quaternionic contact form $\Lambda$ may be written as $\Lambda=\Lambda^II+\Lambda^JJ+\Lambda^KK$, where $\Lambda^S=\zeta^S-\cos^2\eta \,d\phi_S$, $S=I, J, K$. It follows that
\[
\Lambda^S(V_i)= -\frac{\overline{w}_iS}{2(1-\rho^2)}-\frac{d\phi_S(\overline{w}_i\frac{\partial}{\partial \phi})}{2(1-\rho^2)}=0, 
\]
as required. Next, consider a smooth curve $\gamma$ starting at 0 in $\mathbb{H}H^n$:
\[
\gamma(t)=(\gamma_1(t),\cdots,\gamma_n(t)),
\]
where 
\begin{equation*}
\gamma_i=\sum_{S=1,I, J,K}\gamma_i^SS \quad \in \quad C(\R_{\ge0},\mathbb{H}).
\end{equation*} 
Then,
\[
\dot{\gamma}(t)=\mathrm{Re}\sum_{i=1}^n \dot{\gamma}_i(t)\frac{\partial}{\partial w_i}= \frac12\sum_{i=1}^n \left(\dot{\gamma}_i(t)\frac{\partial}{\partial w_i}+\frac{\partial}{\partial \overline{w}_i}\dot{\overline{\gamma}}_i(t)\right).
\]
Consequently, the horizontal lift, $\pmb{\gamma}$, of $\gamma$ at the north pole\footnote{This is the vector with inhomogeneous coordinates $(q_1 = q_2 = \dots = q_n = 0, q_{n+1} = 1)$.} 
is given in cylindrical coordinates $(w,\phi)$ by:
\begin{align*}
\dot{\pmb{\gamma}}(t)=\mathrm{Re}\sum_{i=1}^n \dot{\gamma}_i(t)\left(\frac{\partial}{\partial w_i}+\frac{\overline{w}_i}{2(1-\rho^2)\cos^2\eta}\frac{\partial}{\partial \phi} \right) 
\end{align*}

whence
\begin{align*}
&\dot{w}_i(t)=\dot{\gamma}_i(t), \quad \dot{\overline{w}}_i(t)=\dot{\overline{\gamma}}_i(t),\\
&\dot{\phi}(t)= \frac{1}{2(1-\rho^2)\cos^2\eta(t)}\mathrm{Im}\sum_{i=1}^n \dot{\gamma}_i(t)\overline{w}_i.
\end{align*}
Let $\Theta(t)$ denote the $\mathbf{SU}(2)$-valued path, from \eqref{eq-contact-form} we can easily check that
\begin{equation}\label{eq-Theta-phi-cord}
\Theta(t)^{-1}\dot{\Theta}(t)=\cos^2\eta(t) \,\dot{\phi}(t)=\frac{1}{2(1-\rho^2)}\sum_{i=1}^n\left(\dot{w}_i(t) \overline{w}_i(t) - w_i(t)\dot{\overline{w}}_i(t)\right).
\end{equation}
As a consequence, the $\mathbf{AdS}^{4n+3}(\mathbb{H})$-valued path $\pmb{\gamma}$ is given by
\[
\pmb{\gamma}(t)=\frac{\Theta(t) }{\sqrt{1-|\gamma(t)|^2}} \left( \gamma(t),1\right),
\]
with
\[
\int_0^t\Theta(s)^{-1} \circ d\Theta(s)
 =\int_{\gamma[0,t]} {\zeta}.
\]
Similarly, the horizontal stochastic lift of the Brownian motion $(w(t))_{t \ge 0}$ is 
\begin{equation*}
\frac{\Theta(t)}{\sqrt{1-|w(t)|^2}} \left( w(t),1\right)
\end{equation*}
 with
\[
\int_0^t\Theta(s)^{-1}\circ d\Theta(s)  =\int_{w[0,t]} {\zeta}=\int_{w[0,t]}\circ d\A(t).
\]
\end{proof}

Our next theorem will show that the fiber motion $\Theta(t)$ on the $\mathbf{SU}(2)$-bundle is in fact a time-changed Brownian motion process on $\mathbf{SU}(2)$. To see that, we recall the notions of stochastic exponential (resp. stochastic logarithm) of semi-martingales on a Lie algebra (resp. Lie group) defined as follows (cf \cite{HDL}):
\begin{definition}\label{def-BM-group}
Let $X(t)$ be a semi-martingale on $\mathbf{SU}(2)$ started from the identity and $M(t)$ a semi-martingale on its Lie algebra $\mathfrak{su}(2)$ started from $0$. If these two processes satisfy the Stratonovich differential equation
\[
dX(t)=X(t)\circ dM(t)
\]
then we call $X(t)$ the \emph{stochastic exponential} of $M(t)$ and $M(t)$ the \emph{stochastic logarithm} of $X(t)$. 
In particular if $M(t)$ is a standard $3$-dimensional Brownian motion, then $X(t)$ is a Brownian motion on $\mathbf{SU}(2)$.
\end{definition} 

\begin{theorem}\label{diff1}
Let $r(t)=\arctanh |w(t)|$. The process $\left( r(t), \Theta(t)\right)_{t \ge 0}$ is a diffusion with generator
 \[
L=\frac{1}{2} \left(\frac{\partial^2}{\partial r^2}+((4n-1)\coth r+3\tanh r)\frac{\partial}{\partial r^2}+{\tanh^2 r} \Delta_{\mathbf{SU}(2)} \right).
 \]
 As a consequence the following equality in distribution holds
 \begin{equation}\label{eq-mp-H}
\left( r(t) ,\Theta(t) \right)_{t \ge 0} \overset{d}{=} \left(r(t),\beta_{\int_0^t \tanh^2 r(s)ds}\right)_{t \ge 0},
\end{equation}
where $(\beta_t)_{t \ge 0}$ is a standard Brownian motion process on $\mathbf{SU}(2)$ independent from $r$.
\end{theorem}

\begin{proof}
We first compute, in cylindrical coordinates $(w,\phi)$, the generator of the diffusion $X$ introduced in \eqref{eq-BM-H}. We start with the Laplace-Beltrami operator on $\mathbb{H}H^n$ that writes
\begin{equation}\label{eq-Laplacian-CPn}
\Delta_{\mathbb{H}H^n}= 4(1-\rho^2)\mathrm{Re} \left(\sum_{i=1}^n \frac{\partial^2}{\partial\overline{w_i}\partial w_i} - \overline{\mathcal{R}}\mathcal{R}\right)
\end{equation}
where $\rho=|w|=\tanh r$ and
\[
\mathcal{R}=\sum_{j=1}^n w_j \frac{\partial}{\partial w_j}
\] 
is the quaternionic Euler operator. Since $X$ is the horizontal lift of the Riemannian Brownian motion $w$, its generator is $(1/2)L_{\mathbf{AdS}^{4n+3}(\bH)}$ where $L_{\mathbf{AdS}^{4n+3}(\bH)}$ is the horizontal lift to  $\mathbf{AdS}^{4n+3}(\bH)$ of $\Delta_{\mathbb{H}H^n}$. As we have seen,  the horizontal lift to $\mathbf{AdS}^{4n+3}(\bH)$ of the vector field $\frac{\partial}{\partial w_i}$ is given by $V_i$ in \eqref{eq-V-H},  therefore
\begin{align}\label{genera}
L_{\mathbf{AdS}^{4n+3}(\bH)}&=4(1-\rho^2)\mathrm{Re} \left(\sum_{i=1}^n \frac{\partial^2}{\partial\overline{w_i}\partial w_i} - \overline{\mathcal{R}}\mathcal{R} - \frac{\rho^2}{4(1-\rho^2)\cos^4\eta} \left(\frac{\partial}{\partial \phi}\right)^2 
 \right. 
 \\& + \left. \frac{1}{2\cos^2\eta}\left(\overline{\mathcal{R}}\frac{\partial}{\partial \phi} - \frac{\partial}{\partial \phi}\mathcal{R}\right) \right) \nonumber 
 \\& = \Delta_{\bH H^n}+\frac{\tanh^2r}{\cos^4\eta}\sum_S\frac{\partial^2}{\partial \phi_S^2} + \frac{2}{\cosh^2 r\cos^2\eta} \left(\overline{\mathcal{R}}\frac{\partial}{\partial \phi} - \frac{\partial}{\partial \phi}\mathcal{R}\right). 
\end{align}
Acting on functions depending only on $(r, \phi_I,\phi_J, \phi_K)$, the operator $L_{\mathbf{AdS}^{2n+1}(\bH)}$ reduces to:
\begin{align*}
&\frac{\partial^2}{\partial r^2}+((4n-1)\coth r+3\tanh r)\frac{\partial}{\partial r}+\frac{\tanh^2 r}{\cos^4\eta}\left(\sum_S\frac{\partial^2}{\partial \phi_S^2} \right)\\
&\quad\quad= \frac{\partial^2}{\partial r^2}+((4n-1)\coth r+3\tanh r)\frac{\partial}{\partial r}+{\tanh^2 r} \Delta_{\mathbf{SU}(2)}.
\end{align*}
Note  that ${\cos^2\eta}d\phi_S$, $S=I, J, K$ is the $3$-contact form on $\mathbf{SU}(2)$. The vector fields $\frac1{\cos^2\eta}\frac{\partial}{\partial \phi_S}$ on $\mathbf{SU}(2)$ are in fact the Reeb vector fields of those contact forms. 

We note that $\Theta(t)$ is a $\mathbf{SU}(2)$-valued process satisfying \eqref{eq-Theta-phi-cord}, and $\frac{\tanh^2 r}{\cos^4\eta}\left(\sum_S\frac{\partial^2}{\partial \phi_S^2}\right)$ generates the process $\phi(t)$ such that 
\[
d\phi(t)=\frac{\tanh r}{\cos^2\eta}d\gamma(t),
\]
where $\gamma(t)$ is a standard $3$-dimensional Brownian motion independent of $r(t)$. Hence 
\begin{equation}\label{eq-Theta-A-H}
\Theta(t)^{-1}d\Theta(t)={\cos^2\eta(t)}d\phi(t)={\tanh r}d\gamma(t).
\end{equation}
If we denote by $\beta(t)$ a standard Brownian motion on $\mathbf{SU}(2)$ independent of $r$. From Definition \ref{def-BM-group}, we know that
$
\beta(t)^{-1}d\beta(t)=d\gamma(t).
$
Hence $(r(t), \Theta(t))$ is generated by
\[
\frac12\left(\frac{\partial^2}{\partial r^2}+((4n-1)\coth r+3\tanh r)\frac{\partial}{\partial r}+{\tanh^2 r} \Delta_{\mathbf{SU}(2)} \right),
\]
and 
 \[
\left( r(t), \Theta(t) \right)_{t \ge 0} \overset{d}{=} \left( r(t),\beta_{\int_0^t \tanh^2 r(s)ds}\right)_{t \ge 0}.
\]
\end{proof}

\begin{corollary}\label{cor-a(t)}
Let $r(t)$ and $\A(t)$ be given as previously. Then
\[
\left( r(t), \A(t) \right)_{t \ge 0} \overset{d}{=} \left( r(t),\gamma_{\int_0^t \tanh^2 r(s)ds}\right)_{t \ge 0},
\]
where $\gamma(t)$, $t\ge0,$ is a standard Brownian motion process in $\R^3$.
\end{corollary}
\begin{proof}
This  directly follows from the definition of $\A(t)$ \eqref{eq-Maurer-Cartan} and the equation \eqref{eq-Theta-A-H}.
\end{proof}

\subsection{Characteristic function of the stochastic area and limit theorem}

In this section we study the characteristic function of the stochastic area $\A(t)$.
Let
\[
\mathcal{L}^{\alpha,\beta}=\frac{1}{2} \frac{\partial^2}{\partial r^2}+\left(\left(\alpha+\frac{1}{2}\right)\coth r+\left(\beta+\frac{1}{2}\right) \tanh r\right)\frac{\partial}{\partial r}, \quad \alpha,\beta >-1
\]
be the hyperbolic Jacobi generator. We will denote by $q_t^{\alpha,\beta}(r_0,r)$ the heat kernel with respect to the Lebesgue measure of the diffusion it generates.

Let $\lambda \in [0,\infty)^3$, $r \in [0,+\infty),$ then Corollary \ref{cor-a(t)}  entails:
\begin{align*}
\mathbb{E}\left(e^{i \lambda\cdot \A(t)}\mid r(t)=r\right) & =\mathbb{E}\left(e^{i \lambda\cdot \gamma_{\int_0^t \tanh^2 r(s)ds}}\mid r(t)=r\right) \\
 &=\mathbb{E}\left(e^{- \frac{|\lambda|^2}{2} \int_0^t \tanh^2 r(s)ds}\mid r(t)=r\right) 
\end{align*}
where $|\lambda|^2=\lambda_I^2+\lambda_J^2+\lambda_K^2$ and $r$ is a diffusion whose generator is given by: 
\[
\mathcal{L}^{2n-1,1}=\frac{1}{2} \left( \frac{\partial^2}{\partial r^2}+((4n-1)\coth r+3\tanh r)\frac{\partial}{\partial r}\right),
\]
and started at $0$. 

\begin{theorem}\label{FTCHn}
For $\lambda \in [0,\infty)^3$, $r \in [0,+\infty)$, and $t >0$
\[
\mathbb{E}\left(e^{i \lambda\cdot \A(t)}\mid r(t)=r\right)=\frac{e^{2nt \mu}}{(\cosh r)^{\mu}} \frac{q_t^{2n-1,\mu+1}(0,r)}{q_t^{2n-1,0}(0,r)}.
\]
where $\mu=\sqrt{|\lambda|^2+1}-1$.
\end{theorem}

\begin{proof}
Note
 \begin{align*}
&\mathbb{E}\left(e^{i \lambda\cdot \A(t)}\mid r(t)=r\right)=\mathbb{E}\left(e^{- \frac{|\lambda|^2}{2} \int_0^t \tanh^2 r(s)ds}\mid r(t)=r\right).
\end{align*}
and
\begin{equation}\label{eq-sde-r-H}
dr(t)= \frac{1}{2} \left( (4n-1)\coth r(t)+3\tanh r(t) \right)dt+d\gamma(t),
\end{equation}
where $\gamma$ is a standard Brownian motion. It implies that almost surely have
\begin{align}\label{transient-H}
r(t) \ge \left( 2n -\frac{1}{2} \right) t +\gamma(t),
\end{align}
and thus $r(t)\to +\infty$ almost surely when $t \to \infty$.
Consider now the local martingale given for any $\mu>0$ by
\begin{align*}
D_t& =\exp \left({\mu} \int_0^t \tanh r(s) d\gamma(s) -\frac{\mu^2}{2}  \int_0^t \tanh^2 r(s) ds \right)  \\
 &=\exp \left({\mu} \int_0^t \tanh r(s) dr(s)-\frac{\mu}{2}(4n-1)t -\frac{3{\mu} +{\mu}^2}{2}  \int_0^t \tanh^2 r(s) ds \right) 
\end{align*}
From It\^o's formula, we have
\begin{align*}
\ln \cosh r(t) & =\int_0^t \tanh r(s) dr(s)+\frac{1}{2} \int_0^t \frac{ds}{\cosh^2 r(s)} \\
 &=\int_0^t \tanh r(s) dr(s)-\frac{1}{2} \int_0^t \tanh^2 r(s) ds+\frac{1}{2} t.
\end{align*}
As a consequence, we deduce that
\[
D_t =e^{-2n{\mu} t} (\cosh r(t))^{\mu} e^{- \frac{{\mu}^2+2{\mu}}{2} \int_0^t \tanh^2 r(s)ds}.
\]
It is easy to prove that $D_t$ is a true martingale using the same argument as in \cite{SAWH} Theorem 3.5. 

  Let  $\mathcal{F}$ denote the natural filtration of $r$ and consider the probability measure $\mathbb{P}^{\mu}$ defined by
\[
\mathbb{P}_{/ \mathcal{F}_t} ^{\mu}=D_t \mathbb{P}_{/ \mathcal{F}_t}=e^{-2n{\mu} t} (\cosh r(t))^{\mu} e^{- \frac{{\mu}^2+2{\mu}}{2} \int_0^t \tanh^2 r(s)ds} \mathbb{P}_{/ \mathcal{F}_t}.
\]
We have then for every bounded and Borel function $f$ on $[0,+\infty]$,
\begin{align*}
\mathbb{E}\left(f(r(t))e^{- \frac{{\mu}^2+2{\mu}}{2} \int_0^t \tanh^2 r(s)ds}\right)=e^{2n{\mu} t} \mathbb{E}^{\mu} \left( \frac{f(r(t))}{(\cosh r(t))^{\mu}} \right).
\end{align*}
From Girsanov theorem, the process
\[
\gamma^{\mu}(t)=\gamma(t)-{\mu} \int_0^t \tanh r(s) ds
\]
is a Brownian motion under the probability $\mathbb{P}^{\mu}$. We note that
\begin{equation}\label{eq-rt-H}
dr(t)= \frac{1}{2} \left( (4n-1)\coth r(t)+(2{\mu} +3)\tanh r(t) \right)dt+d\gamma^{\mu}(t).
\end{equation}
Hence  we  have
 \begin{align*}
\mathbb{E}\left(e^{- \frac{\mu^2+2\mu}{2} \int_0^t \tanh^2 r(s)ds}\mid r(t)=r\right)=\frac{e^{2nt \mu}}{(\cosh r)^{\mu}} \frac{q_t^{2n-1,\mu+1}(0,r)}{q_t^{2n-1,0}(0,r)}.
\end{align*}
The proof is complete by letting $\mu=\sqrt{|\lambda|^2+1}-1$.

 \end{proof}

As an immediate corollary of Theorem \ref{FTCHn}, we deduce an expression for  the characteristic function of the stochastic area process.
\begin{corollary}\label{FTh}
For $\lambda \in [0,\infty)^3$ and $t \ge 0$,
\[
\mathbb{E}\left(e^{i \lambda\cdot \A(t)}\right)=e^{2n \mu t}\int_0^{+\infty} \frac{q_t^{2n-1,\mu+1}(0,r)}{(\cosh r)^{\mu}} dr,
\]
where $\mu=\sqrt{|\lambda|^2+1}-1$.
\end{corollary}



We are now in position to prove the following central limit type theorem.

\begin{theorem}\label{LimitCHn}
When $t \to +\infty$, the following convergence in distribution takes place
\[
\frac{\A(t)}{\sqrt{t}} \to \mathcal{N}(0,\mathrm{Id}_3)
\]
where $\mathcal{N}(0,\mathrm{Id}_3)$ is a $3$-dimensional normal distribution with mean 0 and variance matrix $\mathrm{Id}_3$.
\end{theorem}
\begin{proof}
This is a consequence of $r(t)\to +\infty$ almost surely as $t\to+\infty$,
\[
\coth r(t)\to 1,\quad \tanh r(t)\to 1\quad \mbox{a.s.}
\]
hence
\[
\lim_{t\to+\infty}\frac{1}{t}\int_0^t \tanh^2 r(s)ds=1\quad \mbox{a.s.}
\]
Then from Corollary \ref{cor-a(t)}, we have
\[
\lim_{t\to+\infty}\frac{\A(t)}{\sqrt{t}}=\lim_{t\to+\infty}\gamma_{\frac{1}{t}\int_0^t \tanh^2 r(s)ds}=\gamma_1 \quad \mbox{a.s.}
\]
\end{proof}

\subsection{Formula for the density}
In order to invert the Fourier transform displayed in Corollary \ref{FTh}, we need a suitable expression for the heat kernel of the hyperbolic Jacobi operator: 
\begin{equation*}
\mathcal{L}^{(n, \mu)} = \frac{1}{2} \left[\partial_r^2 + ((4n-1)\coth r + (2\mu + 3)\tanh(r)) \partial_r\right], \quad r \geq 0,
\end{equation*}
subject to Neumann boundary condition at $r = 0$. Though the heat kernel of this operator may be expressed through Jacobi functions (\cite{Koor}), we shall derive below another one which not only leads to the sought density but has also the merit to involve the heat kernel of the $4n+1$-dimensional real hyperbolic space. The derivation is a bit technical and for ease of reading, we shall proceed into three steps. More precisely, we shall firstly map the above hyperbolic Jacobi operator into another one by letting it act on functions of the form $r \mapsto f(r)/\cosh^{\mu}(r)$, where $f$ is a smooth test function. Secondly, we shall exploit results in \cite{Int-OM} to derive the heat kernel of the newly-obtained operator: in this step, we follow the lines of Theorem 2 in \cite{Bau-Dem}. Finally, we use known Fourier transforms to obtain the density of the quaternionic stochastic area process. We start with the following straightforward lemma:
\begin{lemma}
Let $f$ be a smooth function on $\mathbb{R}_+$. Then  
\begin{equation}\label{Inter}
 \mathcal{L}^{n, \mu} \left(\frac{f}{\cosh^{\mu}} \right)(r)=  \frac{1}{\cosh^{\mu}(r)} L^{n, \mu}(f) (r)
\end{equation}
where
\begin{equation*}
2L^{n, \mu} := \partial_r^2 + ((4n-1)\coth r + 3\tanh(r)) \partial_r + \frac{\mu(\mu+2)}{\cosh^2(r)}  -\mu(4n+\mu+2).
\end{equation*} 
\end{lemma}
\begin{proof}
Straightforward computations. 
\end{proof}

The operator $2L^{n, \mu} + \mu(4n+\mu+2) + (2n+1)^2$ is an instance of the radial part of the operator $\Delta_{\alpha\beta}$ studied in \cite{Int-OM} with $\alpha = 1+(\mu/2), \beta = 1-\alpha = -\mu/2$ and double complex dimension $2n$ (see p.229 there). 
Using the same reasoning of the proof of Theorem 2 in \cite{Bau-Dem}, we prove the following: 
\begin{proposition}\label{HKR}
Let $f$ be a smooth compactly-supported function in $\mathbb{H}H^n$. Then, the heat semi-group $e^{tL^{n, \mu}}(f)(0)$ reads: 
\begin{multline*}
\frac{e^{-[(2n+1)^2+\mu(4n+\mu+2)]t/2}}{(2\pi)^{2n}\sqrt{2\pi t}} \int_{\bH H^n} f(w) \frac{dw}{(1-|w|^2)^{2n+2}} 
\\ \int_{d(0,w)}^{\infty} dx \sinh(x) K_{\mu}(x,w) \left(\frac{1}{\sinh(x)}\frac{d}{dx}\right)^{2n} e^{-x^2/(2t)}, 
\end{multline*}
where $d(0,w) = r$ is the geodesic distance in $\mathbb{H}H^n$: 
\begin{equation*}
\cosh^2(d(0,w)) = \frac{1}{1-|w|^2},
\end{equation*}
and
\begin{multline*}
K_{\mu}(x,w) := \frac{1}{\cosh(d(z,w))\sqrt{\cosh^2(x)- \cosh^2(d(0,w))}} \\  {}_2F_1\left(\mu+1, -(\mu+1), \frac{1}{2}; \frac{\cosh(d(z,w)) - \cosh(x)}{2\cosh(d(z,w))} \right).
\end{multline*}
\end{proposition}
\begin{proof}
Consider the `switched' wave Cauchy problem associated $\Delta_{\alpha\beta}, \alpha = 1+(\mu/2), \beta = -\mu/2,$ displayed in eq. (1.1) in \cite{Int-OM}. From Theorem 2 in that paper, its solution is given by: 
\begin{equation*}
u(x,z) = \frac{1}{(2\pi)^{2n}}\left(\frac{1}{\sinh(s)}\partial_s\right)^{2n-1} \int_{d(z,w) < |s|} f(w)K_{\mu} (s,z,w) \frac{dw}{(1-|w|^2)^{2n+2}},
\end{equation*}
where $(x,z) \in \mathbb{R} \times \mathbb{H}H^n$ and 
\begin{multline*}
K_{\mu}(x,z,w) := \frac{(1-\overline{\langle z, w\rangle})^{1+(\mu/2)}(1-\langle z, w\rangle)^{-(\mu/2)}}{\cosh(d(0,w))\sqrt{\cosh^2(x)- \cosh^2(d(0,w))}} \\  {}_2F_1\left(\mu+1, -(\mu+1), \frac{1}{2}; \frac{\cosh(d(0,w)) - \cosh(x)}{2\cosh(d(0,w))} \right).
\end{multline*}
Following the proof of Theorem 2 in \cite{Bau-Dem}, we next deduce the heat kernel of $\Delta_{\alpha\beta}$ from $u(x,z)$. 
To this end, we differentiate $x \mapsto u(x,z)$ to get the solution to the `standard' wave Cauchy problem associated $\Delta_{\alpha\beta}$: 
\begin{equation*}
v(x,z) = \partial_x u(x,z) = \frac{\sinh(x)}{(2\pi)^{2n}}\left(\frac{1}{\sinh(x)}\partial_x\right)^{2n} \int_{d(z,w) < |x|} f(w)K(x,z,w) \frac{dw}{(1-|w|^2)^{2n+2}}. 
\end{equation*}
Then, we use the spectral formula (see e.g. \cite{Gri-Nog}):
\begin{equation*}
e^{tL}  = \frac{1}{\sqrt{4\pi t}}\int_{\mathbb{R}}e^{-x^2/(4t)} \cos(x\sqrt{-L}) dx,
\end{equation*}
relating the heat semi-group of a self-adjoint non positive operator $L$ to the solution of its wave Cauchy problem (we wrote the wave propagator as $\cos(x\sqrt{-L})$ which should be understood in the spectral sense). According to this formula and from Proposition 2 in \cite{Int-OM}, we deduce that $\Delta_{\alpha\beta}, \alpha = 1+(\mu/2), \beta = -\mu/2,$ is a non positive self-adjoint operator and that (we perform $2n$ integrations by parts then use Fubini Theorem): 
\begin{multline*}
e^{t\Delta_{\alpha\beta}}(f)(z) = \frac{1}{\sqrt{\pi t}}\int_{0}^{\infty} e^{-x^2/(4t)}v(x,z) dx = \frac{1}{(2\pi)^{2n}\sqrt{\pi t}} 
\\ \int_{\bH H^n} f(w) \frac{dw}{(1-|w|^2)^{2n+2}} 
\int_{d(z,w)}^{\infty} dx \sinh(x) K_{\mu}(x,z,w) \left(-\frac{1}{\sinh(x)}\frac{d}{dx}\right)^{2n} e^{-x^2/(4t)}. 
\end{multline*}
Specializing this formula to $z=0$, we see from the definition of $K_{\mu}(x,z,w)$ that the heat kernel of $e^{t\Delta_{\alpha\beta}}(f)(0)$ is radial. keeping in mind the aforementioned relation between the radial part of $\Delta_{\alpha\beta}$ with the special parameters 
$\alpha = 1+(\mu/2), \beta = -\mu/2,$ and $L^{n, \mu} + \mu(4n+\mu+2) + (2n+1)^2$, the statement of the proposition follows (we simply wrote $K(x,w)$ for $K(x,0,w)$). 
\end{proof}
With the help of proposition \eqref{HKR}, we are ready to derive the density of $\A(t)$. 
\begin{theorem}
Let $s_{t, 4n+1}(\cosh(x))$ be the heat kernel of the $4n+1$-dimensional real hyperbolic space (\cite{Gri-Nog}, \cite{JW}):
\begin{align*}
s_{t,4n+1}(\cosh(x)) = \frac{e^{-(2n)^2t/2}}{(2\pi)^{2n}\sqrt{2\pi t}}\left(\frac{1}{\sinh(x)}\frac{d}{dx}\right)^{2n}e^{-x^2/(2t)},
\end{align*}
and 
\begin{equation*}
 {\it I}_{m-1/2}(u) = \sum_{j \geq 0} \frac{\sqrt{\pi}}{j!\Gamma(j+m+1/2)} \left(\frac{u}{2}\right)^{2j+m-1/2}
\end{equation*}
be the modified Bessel function. Define also the time-dependent symmetric polynomials $Q_{2m}, m \geq 0,$ in $(v_1,v_2,v_3)$ of degree $2m$ by: 
\begin{equation*}
Q_{2m}(v_1,v_2, v_3, t) := e^{|v|^2/(2t)}\left(\Delta_v^m  e^{-|v|^2/(2t)}\right), \quad v \in \mathbb{R}^3,
\end{equation*}
where $\Delta_v$ is the Euclidean Laplacian in $\mathbb{R}^3$ acting on $v$. Then the density of the quaternionic stochastic area process $\A(t)$ is given by: 
\begin{multline*}
\frac{e^{-(4n+1)t/2}}{(2\pi t)^{3/2}}e^{-|v|^2/(2t)}  \int_0^{\infty}\sinh(r)^{4n-1}\cosh^2(r)  \int_{0}^{\infty} du  s_{t, 4n+1}(\cosh(u)\cosh(r)) \\
\sum_{m \geq 0} \frac{(-1)^m}{m!}\left(\frac{u}{2}\right)^{m+1/2} {\it I}_{m-1/2}\left(\frac{u}{2}\right)Q_{2m}(v_1,v_2,v_3, t), \quad v \in \mathbb{R}^3.
\end{multline*}
\end{theorem}

\begin{remark}
The polynomial $Q_{2m}$ may be expressed as a linear combination of products of (even) Hermite polynomials: 
\begin{equation*}
H_j(x):= (-1)^je^{x^2/2}\frac{d^j}{dx^j}e^{-x^2/2}. 
\end{equation*}
Indeed, it suffices to expand: 
\begin{equation*}
\Delta_v^m = \sum_{j_1+j_2+j_3 = m} \frac{m!}{j_1!j_2!j_3!} \partial_{v_1}^{2j_1}\partial_{v_2}^{2j_2}\partial_{v_3}^{2j_3} 
\end{equation*}
to get the representation: 
\begin{equation}\label{Hermite}
Q_{2m}(v_1,v_2,v_3, t) = \frac{1}{t^m} \sum_{j_1+j_2+j_3 = m} \frac{m!}{j_1!j_2!j_3!}H_{2j_1}\left(\frac{v_1}{\sqrt{t}}\right)H_{2j_2}\left(\frac{v_2}{\sqrt{t}}\right)H_{2j_3}\left(\frac{v_3}{\sqrt{t}}\right).
\end{equation}
\end{remark}

\begin{proof}
Since the radial part of the measure (this is the volume measure of $\bH H^n$)
\begin{equation*}
\frac{dw}{(1-|w|^2)^{2n+2}}
\end{equation*}
is $\sinh(r)^{4n-1}\cosh^3(r) dr, |w| = \tanh(r)$, then the intertwining relation \eqref{Inter} together with proposition \ref{HKR} yield:
\begin{multline*}
e^{2n\mu t} \frac{q_t^{2n-1,\mu+1}(0,r)}{\cosh^{\mu}(r)} = \frac{e^{-[(2n+1)^2+\mu(\mu+2)]t/2}}{(2\pi)^{2n}\sqrt{2\pi t}} \sinh(r)^{4n-1}\cosh^3(r)
\\ \int_{d(0,w) = r}^{\infty} dx \sinh(x) K_{\mu}(x,w) \left(\frac{1}{\sinh(x)}\frac{d}{dx}\right)^{2n} e^{-x^2/(2t)}. 
\end{multline*}
Performing the variable change $\cosh(x) = \cosh(u)\cosh(r)$ and using the expression of the heat kernel $s_{t,4n+1}$, we equivalently write:  
\begin{multline*}
e^{2n\mu t} \frac{q_t^{2n-1,\mu+1}(0,r)}{\cosh^{\mu}(r)}  = e^{-2nt - (\mu+1)^2t/2} \sinh(r)^{4n-1}\cosh^2(r) \\ 
\int_{0}^{\infty} du   {}_2F_1\left(-(\mu+1), \mu+1, \frac{1}{2}; \frac{1- \cosh(u)}{2} \right) s_{t, 4n+1}(\cosh(u)\cosh(r)).
\end{multline*}
But, the identity 
\begin{equation*}
{}_2F_1\left(-(\mu+1),  (\mu+1), \frac{1}{2}; \frac{1 - \cosh(u)}{2}\right) = \cosh((\mu+1)u),
\end{equation*}
entails further: 
\begin{multline*}
e^{2n\mu t} \frac{q_t^{2n-1,\mu+1}(0,r)}{\cosh^{\mu}(r)} = e^{-2nt - (\mu+1)^2t/2} \sinh(r)^{4n-1}\cosh^2(r) \\ \int_{0}^{\infty} du  \cosh((\mu+1)u) s_{t, 4n+1}(\cosh(u)\cosh(r)).
\end{multline*}
Consequently, recalling $(\mu+1)^2 = |\lambda|^2+1$, the characteristic function of $\A(t)$ admits the following expression: 
\begin{multline}\label{NewFor}
\mathbb{E}\left(e^{i \lambda\cdot \A(t)}\right) = e^{-(4n+1)t/2} \int_0^{\infty}\sinh(r)^{4n-1}\cosh^2(r) \\ \int_{0}^{\infty} du  e^{-|\lambda|^2t/2} \cosh(\sqrt{|\lambda|^2+1}u) s_{t, 4n+1}(\cosh(u)\cosh(r)).
\end{multline} 
In order to derive the density of $\A(t)$, it suffices to write $e^{-|\lambda|^2t/2} \cosh(\sqrt{|\lambda|^2+1}u)$ as a Fourier transform in the variable $\lambda$ and to apply Fubini Theorem. To this end, we expand: 
\begin{align*}
\cosh(\sqrt{|\lambda|^2+1}u)  & = \sum_{j \geq 0} \frac{u^{2j}}{(2j)!} \sum_{m=0}^j\binom{j}{m}|\lambda|^{2m} 
\\& =  \sqrt{\pi}\sum_{j \geq 0} \frac{u^{2j}}{2^{2j}\Gamma(j+1/2)} \sum_{m=0}^j\frac{1}{m!(j-m)!}|\lambda|^{2m} 
\\& =  \sqrt{\pi}\sum_{m \geq 0} \frac{|\lambda|^{2m}}{m!}\sum_{j \geq m}\frac{1}{(j-m)!}\frac{u^{2j}}{2^{2j}\Gamma(j+1/2)} 
\\&= \sum_{m \geq 0} \frac{|\lambda|^{2m}}{m!}\left(\frac{u}{2}\right)^{m+1/2} {\it I}_{m-1/2}\left(u\right),
\end{align*}
and write
\begin{align*}
|\lambda|^{2m}e^{-|\lambda|^2t/2} & = \frac{(-1)^m}{(2\pi t)^{3/2}} \int_{\mathbb{R}} \left(\Delta_v^m e^{i \lambda \cdot v}\right) e^{-|v|^2/(2t)} dv 
\\& = \frac{(-1)^m}{(2\pi t)^{3/2}} \int_{\mathbb{R}} e^{i\lambda \cdot v} \left(\Delta_v^m  e^{-|v|^2/(2t)}\right) dv 
\\& = \frac{(-1)^m}{(2\pi t)^{3/2}} \int_{\mathbb{R}} e^{i \lambda \cdot v}e^{-|v|^2/(2t)} Q_{2m}(v_1,v_2,v_3, t) dv.
\end{align*}
Using the bound (\cite{Erd}, p.208),
\begin{equation*}
|H_{2j}(x)| \leq e^{x^2/4}2^{2j}j!, 
\end{equation*}
we can see from \eqref{Hermite} that 
\begin{equation*}
|Q_{2m}(v_1, v_2, v_3, t)| \leq  \frac{m!2^{2m}}{t^m} e^{|v|^2/(4t)}\sum_{j_1+j_2+j_3 = m}1 = \frac{m!2^{2m}}{t^m} e^{|v|^2/(4t)}\frac{(m+2)(m+1)}{2}.
\end{equation*}
Combined with the following bound for the modified Bessel function (see e.g. \cite{Erd}, p.14):
\begin{equation*}
{\it I}_{m-1/2}\left(u\right) \leq \left(\frac{u}{2}\right)^{m-1/2} \frac{e^u}{\Gamma(m+1/2)}, 
\end{equation*}
we get:
\begin{multline*}
e^{-|\lambda|^2t/2} \cosh(\sqrt{|\lambda|^2+1}u)  = \frac{1}{(2\pi t)^{3/2}} \int_{\mathbb{R}} e^{i\langle \lambda, v \rangle} e^{-|v|^2/(2t)} \\ 
\sum_{m \geq 0} \frac{(-1)^m}{m!}\left(\frac{u}{2}\right)^{m+1/2} {\it I}_{m-1/2}\left(\frac{u}{2}\right)Q_{2m}(v_1,v_2,v_3),
\end{multline*}
where the series in the RHS is absolutely convergent and is bounded by: 
\begin{equation*}
e^{u+|v|^2/(4t)} \sum_{m \geq 0}\frac{u^{2m}(m+1)(m+2)}{2t^m\Gamma(m+1/2)}. 
\end{equation*}
Plugging this Fourier transform in the RHS of \eqref{NewFor}, we only need to check that Fubini Theorem applies. But, the estimate (\cite{JW}, eq. 3.25):
\begin{equation*}
s_{4n+1}(\cosh(\delta)) \leq C \frac{\delta}{\sinh(\delta)} e^{-\delta^2/(2t)}, \quad C, \delta > 0,
\end{equation*}
together with
\begin{equation*}
\cosh^{-1}[\cosh(u)\cosh(r)] \geq \cosh^{-1}\left[\frac{1}{2}(\cosh(u+r)\right] \geq (r+u), \quad r,u \rightarrow +\infty, 
\end{equation*}
shows that $s_{4n+1}(\cosh(u)\cosh(r)) \leq Ce^{-(r+u)^2/(2t)}$. Hence, Fubini Theorem applies which finishes the proof.
\end{proof}

\section{Stochastic area process on quaternionic projective spaces $\bH P^n$}\label{sec--s-geometry}

\subsection{The Quaternionic Hopf fibration}

We now turn to the study of the quaternionic Hopf fibration, start with some preliminaries, and refer to \cite{BW2} for more details.
As previously, $\mathbb{H}$ is the quaternionic field and $I,J,K \in \mathbf{SU}(2)$ are the Pauli matrices. Define the quaternionic sphere $\bS^{4n+3}$ by: 
\[
\bS^{4n+3}=\lbrace q=(q_1,\cdots,q_{n+1})\in \mathbb{H}^{n+1}, | q |^2 =1\rbrace.
\]
Then, $\mathbf{SU}(2)$ acts on it by isometries and the quotient space $\bS^{4n+3}/ \mathbf{SU}(2)$ can be identified with the quaternionic projective space $\bH P^n$ endowed with its canonical quaternionic K\"ahler metric. 
Besides, the projection map $\bS^{4n+3}\to \bH P^n$ is a Riemannian submersion with totally geodesic fibers isometric to $\mathbf{SU}(2)$, and the fibration
\[
\mathbf{SU}(2)\to \bS^{4n+3}\to \bH P^n
\]
is called the quaternionic Hopf fibration. 

 By analogy with the AdS setting and as in \cite{BW2}, we shall use cylindrical coordinates which are adapted to geometry of the fibration. Let $(w_1,\dots,w_n)$ denote a point on the base space $\bH P^n$, and $(\theta_I, \theta_J, \theta_K)$ be the coordinates for the Lie algebra $\mathfrak{su}(2)$ of traceless skew-Hermitian $2\times 2$ matrices. The cylindrical coordinates are given by the map
\begin{eqnarray*}
\bH P^n \times \mathfrak{su}(2) & \rightarrow & \bS^{4n+3} \\
(w_1,\dots, w_n, \theta_I, \theta_J, \theta_K) & \mapsto &\left(\frac{e^{ I\theta_I +J\theta_J +K\theta_K} w_1}{\sqrt{1+\rho^2}},\cdots,\frac{e^{ I\theta_I +J\theta_J +K\theta_K}w_n}{\sqrt{1+\rho^2}},\frac{e^{ I\theta_I +J\theta_J +K\theta_K}}{\sqrt{1+\rho^2}} \right)
\end{eqnarray*}
where $\rho = \sqrt{\sum_{j=1}^{n}|w_j|^2}$ and $w_i = q_{n+1}^{-1}q_i$, $i=1,\dots, n,$ are inhomogeneous coordinates in $\bH P^n$. \subsection{Stochastic area process on $\mathbb{H}P^n$}
\begin{definition}
Let $(w(t))_{t \ge 0}$ be a Brownian motion on $\mathbb{H}P^n$ started at $0$\footnote{We call $0$ the point with inhomogeneous coordinates $w_1=0,\cdots, w_{n}=0$}. The quaternionic stochastic area process of $(w(t))_{t \ge 0}$ is a process in $\R^3$ defined by
\[
\A(t)=\frac{1}{2}\sum_{j=1}^n \int_0^t \frac{dw_j(s)\overline{w_j}(s)-  w_j(s)d\overline{w_j}(s)}{1+|w(s)|^2},
\]
where the above stochastic integrals are understood in the Stratonovich, or equivalently in the It\^o sense. 
\end{definition}
The following theorem shows that the quaternionic stochastic area process of the Brownian motion on $\mathbb{H}P^n$ can be interpreted as the fiber motion of the horizontal Brownian motion on $\mathbb{S}^{4n+3}$.

 \begin{theorem}\label{horizon-S}
 Let $(w(t))_{t \ge 0}$ be a Brownian motion on $\mathbb{H}P^n$ started at 0, and $(\Theta(t))_{t\ge 0}$ be the solution of the SDE
 \begin{equation}\label{eq-Maurer-Cartan-S}
 d\Theta(t) =-\Theta(t)  \circ d\A(t).
 \end{equation}
 The $\bS^{4n+3}$-valued diffusion process
 \begin{equation}\label{eq-BM-S}
 X_t=\frac{{\Theta(t)} }{\sqrt{1+|w(t)|^2}} \left( w(t),1\right), \quad t \ge 0
 \end{equation}
 is the horizontal lift at the north pole of $(w(t))_{t \ge 0}$ by the submersion $\bS^{4n+3}\to \bH P^n$.
 \end{theorem}
\begin{remark}\label{remark-MC-S}
The SDE \eqref{eq-Maurer-Cartan-S} means that the integral of Maurer-Cartan on $\mathbf{SU}(2)$ along $\Theta(t)$ is exactly given by one half of the stochastic area process $\A(t)$, $t\ge0$.
\end{remark}

 \begin{proof}
The proof parallels the  one in the anti-de Sitter case. The horizontal lift to $\bS^{4n+3}$  of the vector field 
\begin{equation*}
\frac{\partial}{\partial w_i}:=\frac12\left(\frac{\partial}{\partial t_i}-\frac{\partial}{\partial x_i}I-\frac{\partial}{\partial y_i}J-\frac{\partial}{\partial z_i}K\right)
\end{equation*}
is given by 
\begin{equation}\label{eq-V-S}
V_i= \frac{\partial}{\partial w_i} - \frac{\overline{w_i}}{2(1+\rho^2)\cos^2\eta}\frac{\partial}{\partial \phi},
\end{equation}
where 
\begin{equation}\label{eq-phi-S}
\phi:=\frac{\tan\eta}{\eta}\mathfrak{q}=\phi_II+\phi_JJ+\phi_KK,\quad 
\frac{\partial}{\partial \phi} = \frac{\partial}{\partial \phi_I} I + \frac{\partial}{\partial \phi_J} J + \frac{\partial}{\partial \phi_K}K.
\end{equation}

Now, consider a smooth curve $\gamma$ starting at 0 in $\mathbb{H}P^n$:
\[
\gamma(t)=(\gamma_1(t),\cdots,\gamma_n(t)),
\]
where $\gamma_i=(\gamma_i^1,\gamma_i^I, \gamma_i^J, \gamma_i^K)\in C(\R_{\ge0}, \R^4)$. Using quaternionic coordinates
\begin{equation*}
\gamma_i=\sum_{S=1,I, J,K}\gamma_i^SS \quad \in \quad C(\R_{\ge0},\mathbb{H}),
\end{equation*}
we readily have:
\[
\dot{\gamma}(t)=\mathrm{Re}\sum_{i=1}^n \dot{\gamma}_i(t)\frac{\partial}{\partial w_i}=\frac12\sum_{i=1}^n \dot{\gamma}_i(t)\frac{\partial}{\partial w_i}+\frac{\partial}{\partial \overline{w}_i}\dot{\overline{\gamma}}_i(t).
\]
We deduce that the horizontal lift, $\pmb{\gamma}$, of $\gamma$ at the north pole is given in the cylindrical coordinates $(w,\phi)$, by
\begin{align*}
\dot{\pmb{\gamma}}(t)=\mathrm{Re}\sum_{i=1}^n \dot{\gamma}_i(t)\left(\frac{\partial}{\partial w_i}-\frac{\overline{w}_i}{2(1+\rho^2)\cos^2\eta}\frac{\partial}{\partial \phi} \right) 
\end{align*}

Hence
\begin{align*}
&\dot{w}_i(t)=\dot{\gamma}_i(t), \quad \dot{\overline{w}}_i(t)=\dot{\overline{\gamma}}_i(t),\\
&\dot{\phi}(t)= -\frac{1}{2(1+\rho^2)\cos^2\eta(t)}\mathrm{Im}\sum_{i=1}^n \dot{\gamma}_i(t)\overline{w}_i
\end{align*}
Let $\Theta(t)$ denote the $\mathbf{SU}(2)$-valued path issued from identity and satisfies
\[
\Theta(t)^{-1}\dot{\Theta}(t)={\cos^2\eta(t)}\dot{\phi}(t),
\]
then we have
\[
\Theta(t)^{-1}\dot{\Theta}(t)=-\frac1{2(1+\rho^2)}\sum_{i=1}^n\dot{w}_i(t) \overline{w}_i(t) - w_i(t)\dot{\overline{w}}_i(t).
\]
As a consequence,
\[
\pmb{\gamma}(t)=\frac{\Theta(t) }{\sqrt{1-|\gamma(t)|^2}} \left( \gamma(t),1\right),
\]
with
\[
\int_0^t\Theta(s)^{-1}d\Theta(s) =-\int_{\gamma[0,t]} {\zeta}.
\]
Similarly, the lift of the Brownian motion $(w(t))_{t \ge 0}$ is the process
\begin{equation*}
\frac{\Theta(t)}{\sqrt{1-|w(t)|^2}} \left( w(t),1\right),\quad t \geq 0,
\end{equation*}
with
\[
\int_0^t\Theta(s)^{-1}d\Theta(s)  =-\int_{w[0,t]} {\zeta}=-\int_{w[0,t]}\circ d\A(t).
\]
\end{proof}

Next we show that the fiber motion $\Theta(t)$ on the $\mathbf{SU}(2)$-bundle is in fact a time-changed Brownian motion process on $\mathbf{SU}(2)$. 
\begin{theorem}\label{diff1-S}
Let $r(t)=\arctan |w(t)|$. The process $\left( r(t), \Theta(t)\right)_{t \ge 0}$ is a diffusion with generator
 \[
L=\frac{1}{2} \left(\frac{\partial^2}{\partial r^2}+((4n-1)\cot r-3\tan r)\frac{\partial}{\partial r^2}+\tan^2 r \Delta_{\mathbf{SU}(2)} \right).
 \]
 As a consequence the following equality in distribution holds
 \begin{equation}\label{eq-mp-S}
\left( r(t) ,\Theta(t) \right)_{t \ge 0}=\left( r(t),\beta_{\int_0^t \tan^2 r(s)ds}\right)_{t \ge 0},
\end{equation}
where $(\beta_t)_{t \ge 0}$ is a standard Brownian motion process on $\mathbf{SU}(2)$ independent from $r$.
\end{theorem}

\begin{proof}
Here again, the proof is very similar to the anti-de Sitter  case. We first compute the generator of the Markov process $X(t)$ as introduced in \eqref{eq-BM-S}. The Laplace-Beltrami operator on $\mathbb{H}P^n$ is given by
\begin{equation}\label{eq-Laplacian-CPn}
\Delta_{\mathbb{H}P^n}= 4(1+\rho^2)\mathrm{Re} \left(\sum_{i=1}^n \frac{\partial^2}{\partial\overline{w_i}\partial w_i} + \overline{\mathcal{R}}\mathcal{R}\right)
\end{equation}
where $\rho=|w|=\tan r$ and $\mathcal{R}=\sum_{j=1}^n w_j \frac{\partial}{\partial w_j}$
is the quaternionic Euler operator. 
We denote by $\frac{1}{2}L_{\bS^{4n+3}}$ the generator of $X(t)$. Since $X(t)$ is the horizontal lift of $w(t)$, then $L_{\bS^{4n+3}}$ is the horizontal lift to  $\bS^{4n+3}$ of $\Delta_{\mathbb{H}P^n}$. Hence we have
\begin{align}\label{genera}
L_{\bS^{4n+3}}&=4(1+\rho^2)\mathrm{Re} \left(\sum_{i=1}^n \frac{\partial^2}{\partial\overline{w_i}\partial w_i} + \overline{\mathcal{R}}\mathcal{R} - \frac{\rho^2}{4(1-\rho^2)\cos^4\eta} \left(\frac{\partial}{\partial \phi}\right)^2+ 
\frac{1}{2\cos^2\eta}\left(\overline{\mathcal{R}}\frac{\partial}{\partial \phi} - \frac{\partial}{\partial \phi}\mathcal{R}\right) \right) \nonumber \\
& = \Delta_{\bH P^n}+\frac{\tan^2r}{\cos^4\eta}\sum_S\frac{\partial^2}{\partial \phi_S^2} + \frac{2}{\cos^2 r\cos^2\eta} \left(\overline{\mathcal{R}}\frac{\partial}{\partial \phi} - \frac{\partial}{\partial \phi}\mathcal{R}\right), 
\end{align}
We then compute  that $L_{\bS^{4n+3}}$ acts on functions depending only on $(r, \phi_I,\phi_J, \phi_K)$ as
\begin{align*}
&\frac{\partial^2}{\partial r^2}+((4n-1)\cot r-3\tan r)\frac{\partial}{\partial r}+\frac{\tan^2 r}{\cos^4\eta}\left(\sum_S\frac{\partial^2}{\partial \phi_S^2} \right)\\
&\quad\quad= \frac{\partial^2}{\partial r^2}+((4n-1)\cot r-3\tan r)\frac{\partial}{\partial r}+\tan^2 r \Delta_{\mathbf{SU}(2)}
\end{align*}

The last equality comes from the fact that the vector fields $\frac1{\cos^2\eta}\frac{\partial}{\partial \phi_S}$ on $\mathbf{SU}(2)$ are in fact the $3$ Reeb vector fields. Let $\Theta(t)$ be a Brownian motion process on $\mathbf{SU}(2)$ that is generated by $\frac12\Delta_{\mathbf{SU}(2)}$, then clearly $(r(t), \Theta(t))$ is generated by
\[
\frac12\left(\frac{\partial^2}{\partial r^2}+((4n-1)\cot r-3\tan r)\frac{\partial}{\partial r}+{\tan^2 r} \Delta_{\mathbf{SU}(2)} \right).
\]
hence in distribution it holds that 
 \[
\left( r(t), \Theta(t) \right)_{t \ge 0}=\left( r(t),\beta_{\int_0^t \tan^2 r(s)ds}\right)_{t \ge 0},
\]
where $\beta(t)$ is a standard Brownian motion on $\mathbf{SU}(2)$ independent of $r$.
\end{proof}

\begin{corollary}\label{cor-a(t)-S}
Let $r(t)$ and $\A(t)$ be given as previously. In distribution we have that 
\[
\left( r(t), \A(t) \right)_{t \ge 0}=\left( r(t),\gamma_{\int_0^t \tan^2 r(s)ds}\right)_{t \ge 0},
\]
where $\gamma(t)$, $t\ge0$ is a standard Brownian motion process in $\R^3$.
\end{corollary}
\begin{proof}
The proof is the same as in Corollary \ref{cor-a(t)}. To avoid repetition, we omit it here.
\end{proof}

\subsection{Characteristic function of the stochastic area and limit theorem}
We now study the characteristic function of $\A(t)$.
Let $\lambda \in [0,\infty)^3$, $r \in [0,\infty)$ and 
\[
I(\lambda,r)=\mathbb{E}\left(e^{i \lambda\cdot \A(t)}\mid r(t)=r\right).
\]
From Theorem \ref{diff1-S}, we know that
\begin{align*}
I(\lambda,r)& =\mathbb{E}\left(e^{i \lambda\cdot \gamma_{\int_0^t \tan^2 r(s)ds}}\mid r(t)=r\right) \\
 &=\mathbb{E}\left(e^{- \frac{|\lambda|^2}{2} \int_0^t \tan^2 r(s)ds}\mid r(t)=r\right) 
\end{align*}
and $r$ is a diffusion with generator given by:
\[
\mathcal{L}^{2n-1,1}=\frac{1}{2} \left( \frac{\partial^2}{\partial r^2}+((4n-1)\cot r-3\tan r)\frac{\partial}{\partial r}\right)
\]
started at $0$.  More generally, the circular Jacobi generator is defined by:
\[
\mathcal{L}^{\alpha,\beta}=\frac{1}{2} \frac{\partial^2}{\partial r^2}+\left(\left(\alpha+\frac{1}{2}\right)\cot r-\left(\beta+\frac{1}{2}\right) \tan r\right)\frac{\partial}{\partial r}, \quad \alpha,\beta >-1,
\]
and we refer the reader to the Appendix of \cite{SAWH} for further details. We denote by $q_t^{\alpha,\beta}(r_0,r)$ its corresponding transition density with respect to the Lebesgue measure.
\begin{theorem}\label{FThj-S}
For $\lambda \in [0,\infty)^3$, $r \in [0,\pi/2)$, and $t >0$ we have
\begin{equation}\label{eq-ft-cond-S}
\mathbb{E}\left(e^{i \lambda\cdot \A(t)}\mid r(t)=r\right)
 =\frac{e^{-2n \mu t}}{(\cos r)^{\mu}} \frac{q_t^{2n-1,\mu+1}(0,r)}{q_t^{2n-1,1}(0,r)},
\end{equation}
where $\mu=\sqrt{|\lambda|^2+1}-1$.
\end{theorem}

\begin{proof}
Note 
\[
I(\lambda,r) =\mathbb{E}\left(e^{- \frac{|\lambda|^2}{2} \int_0^t \tan^2 r(s)ds}\mid r(t)=r\right) 
 \]
and
\[
dr(t)= \frac{1}{2} \left( (4n-1)\cot r(t)-3\tan r(t) \right)dt+d\gamma(t),
\]
where $\gamma$ is a standard Brownian motion.
Consider the local martingale defined for any $\mu>0$ by
\begin{align*}
D_t& =\exp \left( -\mu \int_0^t \tan r(s) d\gamma(s) -\frac{\mu^2}{2}  \int_0^t \tan^2 r(s) ds \right)  \\
 &=\exp \left( -\mu \int_0^t \tan r(s) dr(s)+\frac{\mu}{2}(4n-1)t -\frac{3\mu +\mu^2}{2}  \int_0^t \tan^2 r(s) ds \right).
\end{align*}
From It\^o's formula, we have
\begin{align*}
\ln \cos r(t) & =-\int_0^t \tan r(s) dr(s)-\frac{1}{2} \int_0^t \frac{ds}{\cos^2 r(s)} \\
 &=-\int_0^t \tan r(s) dr(s)-\frac{1}{2} \int_0^t \tan^2 r(s) ds-\frac{1}{2} t.
\end{align*}
As a consequence, we deduce that
\[
D_t =e^{2n\mu t} (\cos r(t))^{\mu} e^{- (\frac{\mu^2}{2}+\mu) \int_0^t \tan^2 r(s)ds}.
\]
This expression of $D$ implies that almost surely $D_t \le e^{2n\mu t}$ and thus $D$ is a true martingale. Let us denote by $\mathcal{F}$ the natural filtration of $r$ and consider the probability measure $\mathbb{P}^{\mu}$ defined by
\[
\mathbb{P}_{/ \mathcal{F}_t} ^{\mu}=e^{2n\mu t} (\cos r(t))^\mu e^{- (\frac{\mu^2}{2}+{\mu}) \int_0^t \tan^2 r(s)ds} \mathbb{P}_{/ \mathcal{F}_t}.
\]
We have then for every bounded and Borel function $f$ on $[0,\pi /2]$,
\begin{align*}
\mathbb{E}\left(f(r(t))e^{- (\frac{{\mu}^2}{2}+{\mu}) \int_0^t \tan^2 r(s)ds}\right)=e^{-2n{\mu} t} \mathbb{E}^{\mu} \left( \frac{f(r(t))}{(\cos r(t))^{\mu}} \right).
\end{align*}
By Girsanov Theorem, the process defined by
\[
\beta(t)=\gamma(t)+{\mu} \int_0^t \tan r(s) ds
\]
is a Brownian motion under the probability $\mathbb{P}^{\mu}$. Since
\[
dr(t)= \frac{1}{2} \left( (4n-1)\cot r(t)-(2{\mu} +3)\tan r(t) \right)dt+d\beta(t),
\]
the proof is complete by letting $\mu=\sqrt{|\lambda|^2+1}-1$.
\end{proof}

\begin{corollary}\label{cor-theta-S}
For $\lambda \in \mathbb{R}_{\ge0}^3$ and $t \ge 0$,
\[
\mathbb{E}\left(e^{i {\lambda}\cdot \A(t)}\right)=e^{-2n \mu t}\int_0^{\pi /2}  \frac{q_t^{2n-1,\mu+1}(0,r)}{(\cos r)^{\mu}} dr,
\]
where $\mu=\sqrt{|\lambda|^2+1}-1$.
\end{corollary}
\begin{proof}
This is a direct consequence of \eqref{eq-ft-cond-S}.
\end{proof}

We are now in position to prove a central limit type theorem for $\A(t)$. 

\begin{theorem}\label{limit CP}
When $t \to +\infty$, the following convergence in distribution takes place
\[
\frac{\A(t)}{\sqrt t} \to \mathcal{N}(0, 2n \mathrm{Id}_3).
\]
\end{theorem}

\begin{proof}
From Corollary \ref{cor-theta-S} we have for every $t>0$,
\[
\mathbb{E}\left(e^{i \lambda\cdot \frac{\A(t)}{\sqrt t}}\right)=e^{-2n(  \sqrt{t|\lambda|^2+t^2}-t) }\int_0^{\pi /2}  \frac{q_t^{2n-1, \sqrt{\frac{|\lambda|^2}{t}+1} }(0,r)}{(\cos r)^{ \sqrt{\frac{|\lambda|^2}{t}+1} -1}} dr 
\]
Using the formula for $q_t^{n-1, \mu }(0,r)$ which is given in the Appendix of \cite{SAWH}, we obtain by dominated convergence that
\[
\lim_{t \to \infty}  \int_0^{\pi /2}  \frac{q_t^{2n-1, \sqrt{\frac{|\lambda|^2}{t}+1} }(0,r)}{(\cos r)^{ \sqrt{\frac{|\lambda|^2}{t}+1} -1}} dr=  \int_0^{\pi /2}  q_\infty^{2n-1,1}(0,r) dr =1.
\]
On the other hand,
\[
\lim_{t \to \infty}  \sqrt{t|\lambda|^2+t^2}-t =\frac{1}{2} | \lambda|^2,
\]
thus one concludes
\[
\lim_{t \to \infty} \mathbb{E}\left(e^{i \lambda\cdot \frac{\A(t)}{\sqrt t}}\right)=e^{- n | \lambda |^2}.
\]
\end{proof}

\subsection{Formula for the density}
The derivation of the density of $\A(t)$ in this setting is rather direct compared to its AdS analog. Actually, we shall use the explicit expression of the cicular Jacobi heat kernel to give a more elaborate expression of the characteristic function, namely: 
\begin{corollary}[of Corollary \ref{cor-theta-S}]
The characteristic function of  the generalized stochastic area process admits the absolutely-convergent expansion: 
\begin{multline*}
\mathbb{E}\left(e^{i {\lambda}\cdot \A(t)}\right) = e^{-2n\mu t}\sum_{j \geq 0} (-1)^j \frac{(2n)_j}{j!} (2j+2n+\mu+1) e^{-2j(j+2n+\mu+1)t} \\ \frac{\mu(\mu+2)}{4}\frac{\Gamma(j+\mu/2)\Gamma(j+2n+\mu+1)}{\Gamma(j+\mu/2+2n+2)\Gamma(j+\mu+2)},
\end{multline*}
where, as before, $\mu=\sqrt{|\lambda|^2+1}-1$.
\end{corollary}
\begin{proof}
Recall the notation $p_t^{(2n-1, \mu+1)}(0,r)$ for the heat kernel of the circular Jacobi operator $\mathcal{L}^{2n-1,\mu+1}$. Then, it is known that (see the Appendix of \cite{SAWH}): 
\begin{multline}\label{CirJac} 
p_t^{(n-1, \mu+1)}(0,r) = \frac{2}{\Gamma(2n)}[\cos(r)]^{2\mu+3}[\sin(r)]^{4n-1} \\ \sum_{j \geq 0}(2j+2n+\mu+1) e^{-2j(j+2n+\mu+1)t} \frac{\Gamma(j+2n+\mu+1)}{\Gamma(j+\mu+2)} P_j^{(2n-1, \mu+1)}(\cos(2r)). 
\end{multline}
Expanding the Jacobi polynomials (\cite{AAR}): 
\begin{align*}
P_j^{(2n-1, \mu+1)}(\cos(2r)) &= \frac{(2n)_j}{j!} {}_2F_1\left(-j, j+2n+\mu+1, 2n; \frac{1-\cos(2r)}{2}\right)
\\&= \frac{(2n)_j}{j!} \sum_{m=0}^j\frac{(-j)_m(j+2n+\mu+1)_m}{(2n)_mm!} \sin^{2m}(r), 
\end{align*}
where ${}_2F_1$ is the Gauss hypergeometric function, we are led (by the virtue of Corollary \ref{cor-theta-S}) to the following Beta integral: 
\begin{equation*}
\int_0^{\pi/2} \cos^{\mu+3}(r)[\sin(r)]^{4n-1+2m} dr = \frac{\Gamma((\mu/2)+2)\Gamma(2n+m)}{2\Gamma(2n+2+m+\mu/2)}.
\end{equation*}
Consequently, 
\begin{multline}
\int_0^{\pi/2} \cos^{\mu+3}(r)[\sin(r)]^{4n-1} P_j^{(2n-1, \mu+1)}(\cos(2r)) dr  = \Gamma\left(\frac{\mu}{2}+2\right) \frac{\Gamma(2n+j)}{2j!} 
\\ \sum_{m=0}^j\frac{(-j)_m(j+2n+\mu+1)_m}{\Gamma(2n+2+m+\mu/2) m!} = \frac{\Gamma((\mu/2)+2)}{\Gamma(2n+2+\mu/2)} \frac{\Gamma(2n+j)}{2j!} 
\\ {}_2F_1\left(-j, j+2n+\mu+1, 2n + 2 + \frac{\mu}{2}; 1\right) = \frac{\Gamma((\mu/2)+2)}{\Gamma(2n+2+\mu/2)} \frac{\Gamma(2n+j)}{2j!} \\ \frac{j!}{(2n+2+\mu/2)_j}P_j^{(2n+1+\mu/2, \mu/2-1)}(-1)
 = \frac{\Gamma((\mu/2)+2)}{\Gamma(2n+2+\mu/2)} \frac{\Gamma(2n+j)}{2j!} \frac{(-1)^j j!}{(2n+2+\mu/2)_j} \\ P_j^{(\mu/2-1, 2n+1+\mu/2)}(1)
=  \frac{\Gamma((\mu/2)+2)}{\Gamma(2n+2+\mu/2)} \frac{\Gamma(2n+j)}{2j!} \frac{(-1)^j (\mu/2)_j}{(2n+2+\mu/2)_j}
\\ = \frac{\mu(\mu+2)}{4} \frac{\Gamma(2n+j)}{2j!} \frac{(-1)^j \Gamma(j+\mu/2)}{\Gamma(j+2n+2+\mu/2)},
\end{multline}
where we used the symmetry relation $P_j^{(a,b)}(-u) = (-1)^jP_j^{(b,a)}(u)$ and the special value 
\begin{equation*}
P_j^{(a,b)}(1) =\frac{(a+1)_j}{j!}.
\end{equation*}
Keeping in mind \eqref{CirJac}, the corollary is proved. 
\end{proof}
\begin{remark}
At the end of the proof, we simplified with $\Gamma(\mu/2)$ and this is allowed when $\mu \neq 0 \leftrightarrow \lambda \neq 0$. When $\mu = 0$, our computations remain valid and should be understood as a limit when $\mu \rightarrow 0$. 
In this case, the only non vanishing term corresponds to $j=0$ so that 
\begin{equation*}
\lim_{\mu \rightarrow 0} (2n+\mu+1)\frac{(\mu/2)((\mu/2)+1)\Gamma(\mu/2)\Gamma(2n+\mu+1)}{\Gamma(2n+2+\mu/2)}= 1.
\end{equation*}
\end{remark}

Now, in order to invert the characteristic function and recover the density of $\A(t)$, it suffices to express: 
\begin{equation*}
(2j+2n+\mu+1) e^{-2\mu(j+n)t} \frac{\mu(\mu+2)}{4}\frac{\Gamma(j+\mu/2)\Gamma(j+2n+\mu+1)}{\Gamma(j+\mu/2+2n+2)\Gamma(j+\mu+2)},
\end{equation*}
as a Fourier transform in $\lambda$, where we recall the relation $\mu+1 = \sqrt{|\lambda|^2+1}$. To proceed, we write for $j \geq 1$:
\begin{equation*}
\frac{\Gamma(j+\mu/2)\Gamma(j+2n+\mu+1)}{\Gamma(j+\mu/2+2n+2)\Gamma(j+\mu+2)} = \frac{(j+\mu+2n)\dots(j+\mu+2)}{(j+(\mu/2)+2n)\dots(j+\mu/2)},
\end{equation*}
so that
\begin{equation}\label{Coeffic1}
\frac{(2j+2n+\mu+1)\mu(\mu+2)(j+\mu+2n)\dots(j+\mu+2)}{4(j+(\mu/2)+2n+1)\dots(j+\mu/2)}  = 2^{2n} + \sum_{k=0}^{2n+1} \frac{a_{k,n}(j)}{\mu+2j+2k} 
\end{equation}
for some real coefficients $a_{k,n}(j)$. For $j=0$, the same decomposition holds: 
\begin{align}
\frac{\mu(\mu+2)}{4}\Gamma(\mu/2)\frac{(2n+\mu+1)\Gamma(2n+\mu+1)}{\Gamma(\mu/2+2n+2)\Gamma(\mu+2)} & = \frac{\Gamma((\mu/2)+2)\Gamma(2n+\mu+2)}{\Gamma(\mu/2+2n+2)\Gamma(\mu+2)} \nonumber
\\& = 2^{2n} + \sum_{k=1}^{2n+1} \frac{a_{k,n}(0)}{\mu+2k} \nonumber
\\& = 2^{2n} + \sum_{k=0}^{2n+1} \frac{a_{k,n}(0)}{\mu+2k}, \label{Coeffic2}
\end{align}
with $a_{0,n}(0) = 0$. Finally, using the integral
\begin{equation*}
\frac{1}{\mu+ 2j+2k} = \int_0^{\infty} e^{-u\mu} e^{-2(j+k)u} du, 
\end{equation*}
followed by the Fourier transform of the three-dimensional relativistic Cauchy distribution (see \cite{BMR09}, Lemma 2.1 with $d=3, m=1$):
\begin{multline*}
e^{-\mu[u+(2j+2n)t]}  = 2[u+(2j+2n)t] e^{[u+(2j+2n)t]} \left(\frac{1}{2\pi} \right)^2 \\ \int_{\mathbb{R}^3}e^{i\lambda \cdot x}\frac{K_2\left(\sqrt{|x|^2+[u+(2j+2n)t]^2 }\right)}{|x |^2+[u+(2j+2n)t]^2 } dx,
\end{multline*}
we arrive at the following expression for the density of $\A(t)$ (note that the modified Bessel function $K_{\nu}(v)$ is equivalent to $\sqrt{\pi/(2v)}e^{-v}$ at infinity and that that the coefficients $(a_k(j), k =0, \cdots, 2n+2)$ are polynomials in $j$ whose degrees are uniformly bounded by $2n+2$, so that we can use Fubini Theorem to interchange the order of integration).
\begin{theorem}
The density of the quaternionic stochastic area process $\A(t)$ is given by the following absolutely-convergent series: 
\begin{multline*}
\frac{1}{2\pi^2}\sum_{j \geq 0} (-1)^j \frac{(2n)_j}{j!} e^{-2j(j+2n+1)t} \\ 
\left\{2^{2n}(2j+2n)e^{(2j+2n)t}\frac{K_2\left(\sqrt{|x|^2+(2j+2n)^2t^2 }\right)}{|x |^2+(2j+2n)^2t^2 }+ \right.
\\ \left. \sum_{k=0}^{2n+1} a_{k,n}(j)\int_0^{\infty} e^{-2(j+k)u}[u+(2j+2n)t] e^{[u+(2j+2n)t]}\frac{K_2\left(\sqrt{|x|^2+[u+(2j+2n)t]^2 }\right)}{|x |^2+[u+(2j+2n)t]^2}du\right\},
\end{multline*}
where $x \in \mathbb{R}^3$ and the coefficients $(a_{k,n}(j))$ are defined by \eqref{Coeffic1} and \eqref{Coeffic2}.
\end{theorem} 
\begin{remark}
it is clear that the limiting behavior of the density of $\A_t/\sqrt{t}$ is given by the first term $j=0$ in the above series. Moreover, the equivalence  
\begin{equation*}
K_{\nu}(v) \sim \sqrt{\frac{\pi}{2v}}e^{-v}, \quad v \rightarrow +\infty,
\end{equation*}
shows that for any $u \geq 0$, 
\begin{align*}
\lim_{t \rightarrow \infty} t^{3/2} [u+2nt] e^{(u+2nt)}\frac{K_2\left(\sqrt{t|x|^2+(u+2nt)^2}\right)}{t|x |^2+(u+2nt)^2} 
 = \sqrt{\frac{\pi}{2}} \frac{1}{(2n)^{3/2}} e^{-|x|^2/(4n)}. 
\end{align*}
Consequently, the density of $\A_t/\sqrt{t}$ converges as $t \rightarrow \infty$ to 
\begin{equation*}
\frac{1}{(4\pi n)^{3/2}} e^{-|x|^2/(4n)}  \left\{2^{2n} + \sum_{k=0}^{2n+1} a_{k,n}(j)\int_0^{\infty} e^{-2ku}du\right\}.
\end{equation*}
But, substituting $\mu =0$ in \eqref{Coeffic2}, we get 
\begin{equation*}
2^{2n} + \sum_{k=0}^{2n+1} \frac{a_{k,n}(0)}{2k} = 1,
\end{equation*}
whence we recover that (using for instance Scheff\'es Lemma) $\A_t/\sqrt{t}$ converges in distribution to a three-dimensional normal distribution of covariance matrix $2n \mathrm{Id}_3$, as shown in theorem \ref{limit CP}.
\end{remark}

\section{Appendices}
\subsection{The quaternionic contact structure of $\mathbf{AdS}^{4n+3}(\mathbb{H})$ }\label{appendix 1}

There is a quaternionic contact structure on $\mathbf{AdS}^{4n+3}(\mathbb{H})$ that we now describe. 
Consider  the quaternionic  form
\begin{align}\label{eq-alpha-form}
\alpha = \frac12\left( \sum_{i=1}^n (dq_i\, \overline{q_i}-q_i\,d\overline{q_i})- (dq_{n+1}\,\overline{q_{n+1}}-q_{n+1}\,d\overline{q_{n+1}})\right)=\alpha_II+\alpha_JJ+\alpha_KK,
\end{align}
then the triple $(\alpha_I,\alpha_J,\alpha_K)$ gives the quaternionic contact structure.  If we denote
\[
T=- \sum_{i=1}^n \left(q_i\frac{\partial}{\partial q_i}-\frac{\partial}{\partial \overline{q_i}}\overline{q_i}\right)+\left(q_{n+1}\frac{\partial}{\partial q_{n+1}}-\frac{\partial}{\partial \overline{q_{n+1}}}\overline{q_{n+1}}\right),
\]
then $T=T_II+T_JJ+T_KK$, $\alpha(T)=3$ and we can easily find that for any $S, S'\in \{I,J,K\}$,
\[
\alpha_S(T_{S'})=-\delta_{SS'}.
\]

Thus, $T_I$, $T_J$, $T_K$ are the three Reeb vector fields of $\alpha$ and are also Killing vector fields on $\mathbf{AdS}^{4n+3}(\mathbb{H})$. In this way, $\mathbf{AdS}^{4n+3}(\mathbb{H})$ is a negative 3-K contact structure (see \cite{Jel} for a definition of this structure).

Using the cylindric coordinates, we can then rewrite the contact form \eqref{eq-contact-form} as follows:
\begin{align*}
\alpha & = \frac12\left( \sum_{i=1}^n ((q_{n+1}dw_i+dq_{n+1}w_i)\, \overline{w_i}\,\overline{q_{n+1}}-q_{n+1}w_i\,(d\overline{w_i}\,\overline{q_{n+1}}+\overline{w_i}d\overline{q_{n+1}}))- (dq_{n+1}\,\overline{q_{n+1}}-q_{n+1}\,d\overline{q_{n+1}})\right)\\
&=\frac12\left( \sum_{i=1}^n q_{n+1}dw_i\, \overline{w_i}\,\overline{q_{n+1}}-q_{n+1}w_i\,d\overline{w_i}\,\overline{q_{n+1}}\right)
-\frac12\left(1-\rho^2 \right)\left(dq_{n+1}\overline{q_{n+1}}-q_{n+1}d\overline{q_{n+1}}   \right).
\end{align*}
Set ${\frak{q}} := { I\theta_1 +J\theta_2 +K\theta_3}$, then 
\begin{equation*}
q_{n+1}=\frac{e^{\frak{q}}}{\sqrt{1-\rho^2}},  \quad \overline{e^{\frak{q}}}=e^{-{\frak{q}}}, \quad \overline{{\frak{q}}}=-{\frak{q}}={\frak{q}}^{-1}\eta^2, \quad e^{\frak{q}}=\cos\eta+\frac{\sin\eta}{\eta}(\theta_II+\theta_JJ+\theta_KK)
\end{equation*}
where $\eta^2=\theta^2_I+\theta^2_J+\theta^2_K$ is the squared Riemannian distance from the identity in $\mathbf{SU}(2)$. Also, for any quaternion $p\in \mathbb{H}$, the relation $0=d(p\,p^{-1})=pdp^{-1}+dp\cdot p^{-1}$ yields
\[
dp^{-1}=-p^{-1}dp\cdot p^{-1}.
\]
As a result, we easily derive: 
\begin{align}\label{eq-contact-form}
de^{\frak{q}}\cdot e^{-{\frak{q}}}-e^{\frak{q}}\cdot de^{-{\frak{q}}} =2\cos^2\eta\, d\left( \frac{\tan\eta}{\eta}{\frak{q}} \right).
\end{align}
Hence, we can equivalently consider the following one form
\begin{align}\label{Lambda}
\Lambda:=\frac{e^{-{\frak{q}}}\,\alpha\, e^{\frak{q}}}{1-\rho^2}=\frac12\left(\sum_{i=1}^n\left(\frac{dw_i\, \overline{w_i}-w_i\,d\overline{w_i}}{1-\rho^2}\right)-2\cos^2\eta\, d\left( \frac{\tan\eta}{\eta}{\frak{q}} \right)\right)
\end{align}
whose horizontal part 
\begin{equation}\label{QKah}
\zeta := \frac12\sum_{i=1}^n\left( \frac{dw_i\, \overline{w_i}-w_i\,d\overline{w_i}}{1-\rho^2}\right)
\end{equation}
 is the quaternionic K\"ahler form on $\bH H^n$, which in turn induces the following sub-Riemannian metric
\begin{equation}\label{eq-metric}
h_{i\bar{k}}=\frac12\left(\frac{\delta_{i{k}}}{1-\rho^2}+\frac{\overline{w_i}w_k}{(1-\rho^2)^2}\right).
\end{equation}

\subsection{The quaternionic contact structure of $\bS^{4n+3}$}
There is a quaternionic contact structure on $\bS^{4n+3}$ which is compatible with the Hopf fibration structure. It is given in Euclidean coordinates $q=(q_1,\cdots,q_{n+1})$ by the quaternionic one-form:
\begin{align}\label{eq-contact-form-S}
\alpha = \frac12 \sum_{i=1}^{n+1} (dq_i\, \overline{q_i}-q_i\,d\overline{q_i})=\alpha_1I+\alpha_2J+\alpha_3K,
\end{align}
or equivalently by the triple $(\alpha_1,\alpha_2,\alpha_3)$ which is a $3$-dimensional contact form. Moreover, if we denote
\[
T=- \sum_{i=1}^n \left(q_i\frac{\partial}{\partial q_i}-\frac{\partial}{\partial \overline{q_i}}\overline{q_i}\right)
\]
then $T=T_1I+T_2J+T_3K$, $\alpha(T)=3$ and we can easily check that
\[
\alpha_i(T_j)=\delta_{ij}.
\]

Hence $T_1$, $T_2$, $T_3$ are the three Reeb vector fields of $\alpha$ and also Killing vector fields on $\bS^{4n+3}$.
The contact form \eqref{eq-contact-form-S} in the cylindrical coordinates is now given by
\begin{align*}
\alpha =\frac12 \sum_{i=1}^{n+1} q_{n+1}dw_i\, \overline{w_i}\,\overline{q_{n+1}}-q_{n+1}w_i\,d\overline{w_i}\,\overline{q_{n+1}}.
\end{align*}
As previously we denote by ${\frak{q}} := { I\theta_1 +J\theta_2 +K\theta_3}$ a point on  $\mathfrak{su}(2)$. Consider an equivalent one form $\Lambda:=\mathrm{Ad}(e^{-\frak{q}})\alpha$. We then have 
\begin{equation}\label{eq-lambda-S}
\Lambda:=e^{-{\frak{q}}}\,\alpha\, e^{\frak{q}}=\frac12\left(\sum_{i=1}^n\left( \frac{dw_i\, \overline{w_i}-w_i\,d\overline{w_i}}{1+\rho^2}\right)+2\cos^2\eta\, d\left( \frac{\tan\eta}{\eta}{\frak{q}} \right)\right).
\end{equation}
We denote by $\zeta$ the horizontal part of $\Lambda$,
\begin{equation}\label{QKah-S}
\zeta := \frac12\sum_{i=1}^n\left( \frac{dw_i\, \overline{w_i}-w_i\,d\overline{w_i}}{1+\rho^2}\right).
\end{equation}
It indeed is the quaternionic K\"ahler form on $\bH P^n$, and induces the following sub-Riemannian metric
\begin{equation}\label{eq-metric}
h_{i\bar{k}}=\frac12\left(\frac{\delta_{i{k}}}{1+\rho^2}-\frac{\overline{w_i}w_k}{(1+\rho^2)^2}\right).
\end{equation}

\end{document}